\PassOptionsToPackage{prologue,dvipsnames}{xcolor} % see https://github.com/borisveytsman/acmart/issues/406#issuecomment-667180341
\documentclass[hidelinks,onefignum,onetabnum]{siamart250211}

\usepackage{lipsum}
\usepackage{amsfonts}
\usepackage{graphicx}
\usepackage[dvipsnames]{xcolor}
\usepackage{comment}
\usepackage{epstopdf}
\usepackage{algorithmic}
\ifpdf
  \DeclareGraphicsExtensions{.eps,.pdf,.png,.jpg}
\else
  \DeclareGraphicsExtensions{.eps}
\fi

\usepackage{amssymb}
\usepackage{arydshln} % for \hdashline
\usepackage{tikz}
\usepackage{tikz-cd}
\usepackage{stmaryrd} % maps from
\usepackage{tablefootnote} % for \tablefootnote
\usepackage{graphbox}
\usepackage{dashrule}
\usepackage{mathtools} % for \coloneqq
\usepackage[shortlabels]{enumitem}
\usepackage{subcaption}
\usepackage{accents}
\usepackage{pgfplots}
\usepackage{hyperref}
\hypersetup{
    colorlinks=true,
    citecolor=red, 
    linkcolor=blue,
    urlcolor=black
}
\usepgfplotslibrary{groupplots}
\pgfplotsset{compat=1.18}

\pgfplotscreateplotcyclelist{igrcolors}{
    {blue, line width=1.1pt},
    {red!80!black, line width=1.1pt},
    {green!50!black, line width=1.1pt}
}
\usepackage{cite}

\usepackage{aliascnt}
\newaliascnt{remark}{theorem}
\newtheorem{remark}[remark]{Remark}
\aliascntresetthe{remark}
\crefname{remark}{Remark}{Remarks}
\Crefname{remark}{Remark}{Remarks}

\newaliascnt{example}{theorem}

\aliascntresetthe{example}
\crefname{example}{Example}{Examples}
\Crefname{example}{Example}{Examples}

\newaliascnt{assumption}{theorem}
\newtheorem{assumption}[assumption]{Assumption}
\aliascntresetthe{assumption}
\crefname{assumption}{Assumption}{Assumptions}
\Crefname{assumption}{Assumption}{Assumptions}

\crefname{section}{section}{sections}
\Crefname{section}{Section}{Sections}

\crefname{subsection}{section}{sections}
\Crefname{subsection}{Section}{Sections}

\crefname{subsubsection}{section}{sections}
\Crefname{subsubsection}{Section}{Sections}

\crefname{appendix}{Appendix}{Appendices}
\Crefname{appendix}{Appendix}{Appendices}

\crefformat{equation}{(#2#1#3)}
\Crefformat{equation}{(#2#1#3)}

\crefrangeformat{equation}{(#3#1#4)--(#5#2#6)}
\Crefrangeformat{equation}{(#3#1#4)--(#5#2#6)}

\crefmultiformat{equation}{(#2#1#3)}{ and (#2#1#3)}
  {, (#2#1#3)}{, and (#2#1#3)}
\Crefmultiformat{equation}{(#2#1#3)}{ and (#2#1#3)}
  {, (#2#1#3)}{, and (#2#1#3)}

\renewcommand*{\dot}[1]{%
  \accentset{\mbox{\large\bfseries .}}{#1}}
\renewcommand*{\ddot}[1]{%
  \accentset{\mbox{\large\bfseries .\hspace{-0.25ex}.}}{#1}}

\headers{IGR Shocks for Compressible Euler}{W. Barham, B. K. Tran, B. S. Southworth and F. Schäfer}

\title{Shock solutions for the one-dimensional information geometric regularization of compressible flow\thanks{\funding{WB was supported by the Director's Fellowship at Los Alamos National Laboratory, project number 20251151PRD1. BKT was supported by Subcontract C6818 issued under NNSA Contract No. 89233218CNA000001. BSS was supported by the DOE Office of Advanced Scientific Computing Research Applied Mathematics program through Contract No. 89233218CNA000001. FS was supported by the Air Force Office of Scientific Research under award number FA9550-23-1-0668 (Information Geometric Regularization for Simulation and Optimization of Supersonic Flow), the Predictive Science Academic Alliance Program (PSAAP Award DE-NA0004261 - “The Center for Information Geometric Mechanics and Optimization
(CIGMO)”) managed by the NNSA (National
Nuclear Security Administration) Office of Advanced Simulation, and the Alfred P. Sloan Foundation via a Sloan Research Fellowship in Mathematics.}
}
}

\author{
    William Barham\thanks{Los Alamos National Laboratory, Theoretical Division, Los Alamos, NM 87545}
    \and
    Brian K. Tran\thanks{University of Colorado Boulder, Department of Applied Mathematics, Boulder, CO 80309. \\ California State University, Long Beach, Department of Mathematics and Statistics, Long Beach, CA 90840.}
    \and
    Ben S. Southworth\footnotemark[2]
    \and
    Florian Schäfer\thanks{New York University, Courant Institute of Mathematical Sciences, New York, NY 10012}
}

\usepackage{amsopn}

\makeatletter
\newcommand{\subalign}[1]{%
  \vcenter{%
    \Let@ \restore@math@cr \default@tag
    \baselineskip\fontdimen10 \scriptfont\tw@
    \advance\baselineskip\fontdimen12 \scriptfont\tw@
    \lineskip\thr@@\fontdimen8 \scriptfont\thr@@
    \lineskiplimit\lineskip
    \ialign{\hfil$\m@th\scriptstyle##$&$\m@th\scriptstyle{}##$\hfil\crcr
      #1\crcr
    }%
  }%
}
\makeatother
\ifpdf
\hypersetup{
  pdftitle={Shock solutions for the one-dimensional information geometric regularization of compressible flow},
  pdfauthor={W. Barham, B. K. Tran, B. S. Southworth and F. Schäfer}
}
\fi

\begin{document}

\maketitle
\renewcommand{\thefootnote}{}
\footnotetext{\scriptsize{E-mail: \texttt{wbarham@lanl.gov}, \texttt{brian.tran@csulb.edu}, \texttt{southworth@lanl.gov}, \texttt{florian.schaefer@nyu.edu}}}
\renewcommand{\thefootnote}{\arabic{footnote}}

% REQUIRED
\begin{abstract}
    The information geometric regularization (IGR) is an inviscid regularization of the compressible Euler equations that alters the geometry of Lagrangian characteristics to prevent trajectories from crossing in finite time. Previous work on IGR established global strong solutions in one dimension, explored thermodynamic effects of the model, and enabled large-scale simulations of compressible flow.
    However, a fundamental question that remains is how this regularization alters the structure and 
    regularity of a shock-like solution.
    
    We prove existence, uniqueness modulo translation, and regularity of transonic compressive IGR shock profiles in one spatial dimension. The analysis applies to the full thermodynamic compressible Euler--IGR model with a general equation of state, subject to mild convexity hypotheses. A traveling-wave ansatz reduces the Euler--IGR equations to a degenerate second-order scalar equation for the density profile. At the sonic crossing, the elliptic coefficient degenerates: the density profile remains continuous, but its derivative diverges. The profile is a classical solution away from this single point, while at the degeneracy it retains quantified H\"older and Sobolev regularity. We also analyze the vanishing-regularization limit, showing that the shock width scales like $\sqrt{\alpha}$ and that the IGR profiles converge to the entropy-admissible Euler shock. 
    %\wjb{note the abstract is 244 words right now; the maximum is 250}
\end{abstract}

% REQUIRED

% REQUIRED

\tableofcontents

\section{Introduction}

%-------------------------------------------------------------------------------------------
\subsection{Shock waves in the Euler equations}

Nonlinear steepening is a generic feature of the one-dimensional compressible Euler equations causing smooth initial data to lose regularity in finite time. In particular, compressive waves can sharpen into shock discontinuities while the conserved quantities themselves remain bounded, and classical solutions then cease to exist. Across a moving discontinuity, conservation of mass, momentum, and energy yields the Rankine--Hugoniot jump conditions \cite{courant1948supersonic, rankine1870thermodynamic, hugoniot1887propagation}. These relations are necessary for a discontinuity to define a weak solution, but they do not by themselves select the physically relevant shock. One therefore imposes an admissibility criterion, typically expressed through entropy production; for the compressive shocks considered here, this is reflected in the Lax compressivity condition \cite{lax1957hyperbolic, smoller1994shock}. In a shock-attached frame, this is the familiar transonic picture in which the flow passes from a supersonic upstream state to a subsonic downstream state.

A complementary point of view replaces the discontinuity by a continuous internal layer connecting the end states and then studies the limit as the regularization vanishes. When this vanishing-regularization limit produces a weak solution satisfying the entropy inequality, one recovers the admissible Euler shock \cite{lax1957hyperbolic, toro2009riemann, riemann_fluid_flow, dafermos2026hyperbolic}. Classical parabolic regularizations provide a standard route to such entropy solutions \cite{LeVeque_2002}, 
while nondispersive conservative regularizations can yield continuous, weakly singular shock profiles
\cite{clamond2018non,pu2018weakly, guelmame2022hamiltonian, guelmame2024hamiltonian}. In this paper, we consider the analogous program for the information geometric regularization (IGR) of the compressible Euler equations \cite{cao2023information}, which smooths shock fronts without introducing viscous dissipation. 

%-------------------------------------------------------------------------------------------
\subsection{Information geometric regularization (IGR)\nopunct}

% \wjb{Rewrite with more subtlety...}

% In the geometric interpretation of the compressible Euler equations, shock formation corresponds to the loss of invertibility of the deformation map transporting the fluid, or equivalently, to the emergence of transported singular measures \cite{cao2023information}. IGR prevents this collapse by introducing a logarithmic barrier associated with the negative entropy functional, which penalizes concentration of mass and enforces strict feasibility within the information geometric manifold of admissible states \cite{cao2023information, cao2024information}. From this perspective, the regularization does not act as a viscous diffusion that oversmooths the solution. Instead, it prevents the formation of a transported singular measure while allowing a weaker singularity to persist in the profile. Our analysis shows that transonic compressive IGR shocks exhibit exactly this intermediate behavior: the density profile is generally not $C^1$ at the sonic crossing and its derivative diverges; however, the singularity is sufficiently mild that the shock does not transport a singular measure. 

IGR mitigates shock formation by modifying the geometry of Lagrangian particle motion so that trajectories are prevented from crossing and instead asymptotically approach one another.
In effect, this results in the introduction of an additional nonlocal pressure force, the entropic pressure, which counteracts the concentration mechanism that drives shock formation \cite{cao2023information,cao2024information}. 
IGR foregoes viscous diffusion that smooths the solution over time. 
Thus, it preserves the fine-scale structures damped by traditional parabolic regularizations \cite{cao2023information,wilfong2025simulatingmanyenginespacecraftexceeding,radhakrishnan2026shocks}. 
Instead, it prevents the onset of the most singular transport pathology while still allowing a weaker singularity to persist in the shock profile. Our analysis shows that transonic compressive IGR shocks exhibit exactly this intermediate behavior: the density profile remains continuous but is generally not $C^1$ at the sonic crossing, and its derivative diverges there; nevertheless, the singularity is mild enough that the shock does not transport a singular measure.

%-------------------------------------------------------------------------------------------
\subsection{Traveling-wave reduction}

To analyze IGR shock layers, we seek solutions depending only on the
shock-frame coordinate $\xi=x-ct$. The traveling-wave ansatz
\cite{dafermos2026hyperbolic} reduces the PDE to an autonomous system of ODEs
for the profile. 
The conserved fluxes allow us to eliminate the remaining
variables, yielding a single second-order equation for the density profile. This
equation has degenerate elliptic character, with a quadratic dependence on the
profile slope. 
Its highest-order coefficient degenerates at the sonic state,
where the profile passes from supersonic to subsonic flow. This 
degeneracy is the main analytic difficulty in the paper. See
\Cref{fig:igr_tw_profile} for a numerically computed IGR shock profile; details
of the computation are given in \Cref{appendix:numerics}.

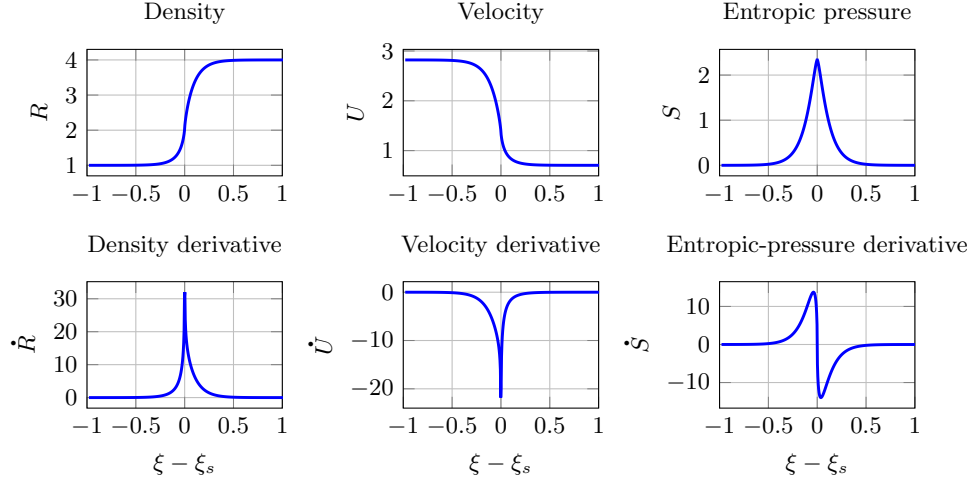
\begin{figure} 
\centering 
\IfFileExists{figures/data/igr_tw_profile.csv}{%
  \def\igrprofilecsv{figures/data/igr_tw_profile.csv}%
}{%
  \def\igrprofilecsv{data/igr_tw_profile.csv}%
}

\begin{tikzpicture}
\begin{groupplot}[
    group style={
        group size=3 by 2,
        horizontal sep=1.6cm,
        vertical sep=1.4cm,
    },
    width=0.32\textwidth,
    height=0.25\textwidth,
    xmin=-1, xmax=1,
    grid=both,
    cycle list name=igrcolors,
    tick label style={font=\small},
    label style={font=\small},
    title style={font=\small},
]

\nextgroupplot[
    ylabel={$R$},
    title={Density},
]
\addplot table [x=xi_centered, y=R, col sep=comma] {\igrprofilecsv};

\nextgroupplot[
    ylabel={$U$},
    title={Velocity},
]
\addplot table [x=xi_centered, y=U, col sep=comma] {\igrprofilecsv};

\nextgroupplot[
    ylabel={$S$},
    title={Entropic pressure},
]
\addplot table [x=xi_centered, y=S, col sep=comma] {\igrprofilecsv};

\nextgroupplot[
    ylabel={$\dot R$},
    xlabel={$\xi-\xi_s$},
    title={Density derivative},
]
\addplot table [x=xi_centered, y=R_xi, col sep=comma] {\igrprofilecsv};

\nextgroupplot[
    ylabel={$\dot U$},
    xlabel={$\xi-\xi_s$},
    title={Velocity derivative},
]
\addplot table [x=xi_centered, y=U_xi, col sep=comma] {\igrprofilecsv};

\nextgroupplot[
    ylabel={$\dot S$},
    xlabel={$\xi-\xi_s$},
    title={Entropic-pressure derivative},
]
\addplot table [x=xi_centered, y=S_xi, col sep=comma] {\igrprofilecsv};

\end{groupplot}
\end{tikzpicture} 
\caption{Numerically computed IGR shock wave profile for $\alpha=0.01$.} 
\label{fig:igr_tw_profile} 
\end{figure}

On either side of the sonic crossing, the first-order formulation of the density
equation has a Riccati-type structure. This observation motivates the change of
variables used in the analysis: once monotonicity of the density has been
established on each branch, density can be used as the independent variable,
which linearizes the quadratic slope dependence and yields an explicit
one-sided solution formula.

%-------------------------------------------------------------------------------------------
\subsection{Paper overview}

%The information geometric regularization (IGR) provides a conservative regularization of the compressible Euler equations through the introduction of an entropic pressure. In contrast to classical viscous regularizations, we will show that the resulting shock profiles exhibit a novel sonic degeneracy associated with the transonic condition and the flux relations. At the sonic density, the density profile remains continuous while its derivative diverges. Despite this divergence, the degeneracy suppresses the flux term appearing in the traveling-wave equation, allowing the profile to cross the sonic state without generating a singular measure defect in the weak formulation. This structure allows the two one-sided solutions to be glued together to form a global weak solution with quantitative Hölder regularity at the sonic crossing and global Sobolev regularity. This sonic degeneracy is a unique feature of the IGR traveling-wave equations, owing to the IGR coupling ansatz of the entropic pressure which is obtained by an elliptic solve \cite{cao2023information, cao2024information}, with the corresponding elliptic operator along a shock degenerating at the sonic point. This behavior contrasts with classical viscous or dispersive regularizations, for which shock profiles are smooth.

The goal of this work is to rigorously establish compressive shock solutions for the one-dimensional IGR model. In particular, we prove existence, uniqueness modulo translation, and regularity of compressive traveling-wave shock profiles connecting states that satisfy the Rankine--Hugoniot relations \cite{courant1948supersonic} and the Lax entropy conditions \cite{lax1957hyperbolic}. 

We proceed in several steps. We first formulate a weak version of the IGR shock equation and introduce monotone admissibility. We then prove that admissible monotone profiles are strictly monotone, which permits a change of variables using density as the independent coordinate. This reduces the profile equation to a linear equation for the squared slope. The resulting one-sided formulas are used to analyze the sonic degeneracy, prove regularity, and glue the two branches into a global weak profile.

In \Cref{sec:TW-main}, we derive the traveling-wave equations for the one-dimensional IGR model, identify the sonic degeneracy, and state the shock admissibility criteria. The analysis applies to the full thermodynamic compressible Euler equations with a general equation of state, subject to mild assumptions discussed in \Cref{appendix:barotropic_eos,appendix:eos-alt-criterion}.

In \Cref{sec:structure-shock}, we study the reduced phase portrait, prove strict monotonicity and nonsonic regularity of admissible profiles, and derive the one-sided formulas for the squared slope.

In \Cref{sec:global}, we establish global existence, regularity, and uniqueness of the IGR traveling-wave solution. Although the density derivative diverges at the sonic state, the profile crosses continuously, and the two one-sided branches glue into a global weak solution with H\"older regularity at the sonic crossing and global Sobolev regularity.

In \Cref{sec:asymptotics}, we analyze the small-$\alpha$ regime and show that the width of the internal layer scales like $\sqrt{\alpha}$. Consequently, the IGR shock profile collapses to a discontinuity as $\alpha\to0$, and the traveling-wave solutions converge to the entropy-admissible Euler shock determined by the Rankine--Hugoniot relations and the Lax compressive conditions \cite{lax1957hyperbolic, courant1948supersonic}.

Finally, in \Cref{appendix:numerics}, we describe the numerical method used to compute the IGR shock profiles shown in the numerical results.

%============================================================================================
\section{A traveling-wave equation for IGR shocks}\label{sec:TW-main}

\sloppy Consider the one-dimensional information geometric regularization (IGR) of the compressible Euler equations \cite{cao2023information, cao2024information}
\begin{subequations}\label{eq:IGR}
\begin{align}
\partial_t \rho + (\rho u)_x &= 0 \,, \\
\partial_t (\rho u) + \big(\rho u^2 + \tilde{p}(\rho,e) + \Sigma\big)_x &= 0 \,, \\
\partial_t E + \big((E + \tilde{p}(\rho,e) + \Sigma)u\big)_x &= 0 \,, \\
\rho^{-1}\Sigma - \alpha(\rho^{-1}\Sigma_x)_x &= 2\alpha u_x^2 \,, \label{eq:IGR-d}
\end{align}
\end{subequations}
where $\rho,u,E,e$ are the fluid density, velocity, total energy density, and specific internal energy, respectively.
Furthermore, $\alpha>0$ is the IGR regularization parameter and $\Sigma$ denotes the \emph{entropic pressure}
introduced by the IGR closure. Note that we consider the IGR regularization with the entropic pressure additively combined with the physical pressure, 
as done in \cite{wilfong2025simulatingmanyenginespacecraftexceeding}, although other couplings are possible \cite{higr, codes_regularization, igr_thermo}. The physical pressure
$\tilde{p}=\tilde{p}(\rho,e)$ is prescribed by an equation of state, with total energy $
E=\rho e+\frac12\rho u^2$, so that $e=\frac{E}{\rho}-\frac12 u^2$.
For generality, we leave the precise form of $\tilde{p}(\rho,e)$ unspecified, although below we will impose mild assumptions ensuring that the resulting IGR shock profile equation is well-behaved.

We seek traveling-wave solutions of the form $
\rho(x,t)=R(\xi)$, $u(x,t)=U(\xi)$, $E(x,t)=\mathcal{E}(\xi)$, $\Sigma(x,t)=S(\xi)$ where $\xi:=x-ct$, satisfying the far-field conditions $R(\xi)\to R_\pm$, $U(\xi)\to U_\pm$, $\mathcal{E}(\xi)\to \mathcal{E}_\pm$, $S(\xi)\to 0$ as $\xi\to\pm\infty$.

\begin{remark}
    We require $R(\xi) > 0$, since $R$ represents the mass density. Therefore, we also require that $R_\pm > 0$. 
\end{remark}
Substituting $\partial_t=-c\partial_\xi$ and $\partial_x=\partial_\xi$ into \Cref{eq:IGR} yields
\begin{subequations}\label{eq:TW}
\begin{align}
\tag{TW$a$}-c\frac{d}{d\xi}R+\frac{d}{d\xi}(RU) &= 0 \,, \label{eq:TWa} \\
\tag{TW$b$}-c\frac{d}{d\xi}(RU)+\frac{d}{d\xi}\big(RU^2+\tilde{p}(R,e)+S\big) &=0 \,, \label{eq:TWb} \\
\tag{TW$c$}-c\frac{d}{d\xi}\mathcal{E} + \frac{d}{d\xi}\big((\mathcal{E} + \tilde{p}(R,e) + S)U\big) &= 0 \,, \label{eq:TWc} \\
\tag{TW$d$}R^{-1}S-\alpha \frac{d}{d\xi}\left(R^{-1}\frac{d}{d\xi}S\right)-2\alpha \left( \frac{d}{d\xi}U\right)^2 &=0 \,. \label{eq:TWd}
\end{align}
\end{subequations}
Throughout, we use prime notation for derivatives with respect to $R$ and dot notation for
derivatives with respect to $\xi$:
$\frac{d}{dR}f(R)=f'(R)$, $\frac{d}{d\xi}g(\xi)=\dot{g}(\xi)$.
Then the chain rule reads $\frac{d}{d\xi}(f\circ g)=f'(g)\dot{g}$.

\paragraph{First integrals and the TW energy curve}
Integrating \Crefrange{eq:TWa}{eq:TWc} gives constant mass, momentum, and energy fluxes, respectively:
\begin{equation}\label{eq:fluxes}
R(U-c)=:m,\qquad
mU+\tilde{p}(R,e)+S=:q,\qquad
(\mathcal{E}+\tilde{p}(R,e)+S)U-c\mathcal{E}=:k.
\end{equation}
In particular,
\begin{equation}\label{eq:U-from-R}
U=c+\frac{m}{R},
\qquad
S=q-mc-\frac{m^2}{R}-\tilde{p}(R,e).
\end{equation}
Using $\tilde{p}(R,e)+S=q-mU$ in the energy flux in \Cref{eq:fluxes} gives
$(\mathcal{E}+q-mU)U-c\mathcal{E}=k$.

In the genuine shock regime considered here, we assume the corresponding mass flux is nonzero, $m \neq 0$. Since $\mathcal{E}=Re+\tfrac12 RU^2$ and $U=c+m/R$, this determines an explicit internal-energy curve $e=e(R)$:
\begin{equation}\label{eq:e-explicit}
    e(R)=e_0+\frac{A}{R}+\frac{m^2}{2R^2},
    \qquad
    e_0:=\frac{k-qc+\frac{m}{2}c^2}{m},
    \qquad
    A:=mc-q.
\end{equation}
We then define the \emph{reduced pressure} along the traveling-wave curve by
\begin{equation}\label{eq:reduced-pressure}
    p(R):=\tilde{p}\big(R,e(R)\big).
\end{equation}
With this definition, the TW relations \Cref{eq:fluxes} reduce to the same algebraic structure as in the barotropic setting. Ensuring convexity of $p$ on the density interval traversed by the profile yields restrictions on admissible equations of state; see \Cref{appendix:barotropic_eos,appendix:eos-alt-criterion}. In the remainder of the paper we assume these criteria hold.
Hence, along any profile,
$U(\xi)=c+\frac{m}{R(\xi)}$ and 
$S(\xi)=q-mc-\frac{m^2}{R(\xi)}-p(R(\xi)).$
These identities motivate the scalar profile equation derived next; the associated Rankine--Hugoniot relations are recorded in \Cref{sec:igr-admissibility}.

%-------------------------------------------------------------------------------------------
\subsection{The IGR shock equation}\label{sec:igr-shock-equation}

Fix fluxes $(c,m,q)$ and the reduced pressure $p(R)$ induced by the traveling-wave energy curve \Cref{eq:e-explicit} via \Cref{eq:reduced-pressure}. It is convenient to record the entropic pressure along a profile as a function of density:
\begin{equation}\label{eq:stress-function}
    \mathcal{S}(R):=q-mc-\frac{m^2}{R}-p(R),
\end{equation}
so that
\begin{equation}\label{eq:S-as-function-of-R}
    S(\xi)=\mathcal{S}(R(\xi)).
\end{equation}
Using $U(\xi)=c+m/R(\xi)$ from \Cref{eq:U-from-R} and differentiating \Cref{eq:S-as-function-of-R} gives
$\dot{U}(\xi)=-\frac{m}{R(\xi)^2}\,\dot{R}(\xi),$
$\dot{S}(\xi)=\mathcal{S}'(R(\xi))\,\dot{R}(\xi),$ and 
$\dot{U}(\xi)^2=\frac{m^2}{R(\xi)^4}\,\dot{R}(\xi)^2$.
Substituting these identities into the IGR closure \Cref{eq:TWd} yields a scalar second-order ODE for the density profile:
\begin{equation}\label{eq:R-ODE-general}
    \frac{\mathcal{S}(R)}{R}
    -\alpha\frac{d}{d\xi}\Biggl(\frac{\mathcal{S}'(R)}{R}\,\dot{R}\Biggr)
    -2\alpha\,\frac{m^2}{R^4}\dot{R}^2=0.
\end{equation}
The traveling-wave problem is to find a heteroclinic solution $R(\xi)$ of \Cref{eq:R-ODE-general} such that $R(\xi)\to R_\pm$, $\dot{R}(\xi)\to 0$ as $\xi\to\pm\infty$,
and then reconstruct $U$ and $S$ from \Cref{eq:U-from-R,eq:S-as-function-of-R}. In the full thermodynamic compressible Euler--IGR setting, one may additionally recover $e(\xi)=e(R(\xi))$ from \Cref{eq:e-explicit} and
$\mathcal{E}(\xi)=R(\xi)e(\xi)+\tfrac12 R(\xi)U(\xi)^2$.

\subsubsection{Degeneracy at a sonic density}\label{sec:degeneracy-sonic-density}
Note that the coefficient of the second-order term in \Cref{eq:R-ODE-general} is proportional to $\mathcal{S}'(R)/R$, where $\mathcal{S}'(R)=\frac{m^2}{R^2}-p'(R)$.
Then, the profile in \Cref{eq:R-ODE-general} is degenerate at densities $R_s$ such that $\mathcal{S}'(R_s)=0$, equivalently,
\begin{equation}\label{eq:sonic-density}
\mathcal{S}'(R_s)=0
\quad\Longleftrightarrow\quad
\frac{m^2}{R_s^2}=p'(R_s)
\quad\Longleftrightarrow\quad
(U|_{R=R_s}-c)^2=p'(R_s).
\end{equation}
We refer to such $R_s$ as \emph{sonic densities}. We now establish the uniqueness of this sonic density under suitable assumptions on the pressure.

\paragraph{Lax/transonic sign change}
The analysis below is formulated directly in terms of the sign conditions
\begin{equation}\label{eq:transonic-sign}
\frac{m^2}{R_-^2}-p'(R_-)>0,\qquad
\frac{m^2}{R_+^2}-p'(R_+)<0.
\end{equation}
For the compressive $1$-shock branch considered here, the reduced barotropic Euler characteristic speed is
$\lambda_1(R,U)=U-\sqrt{p'(R)}$,
and the $1$-shock Lax inequalities \cite{lax1957hyperbolic} read $
\lambda_1(R_-,U_-)>c>\lambda_1(R_+,U_+)$.

Since $U_\pm-c=m/R_\pm$ and the left inequality implies $m>0$, these inequalities are equivalent to \Cref{eq:transonic-sign}.
In particular, by continuity of $p'$ there exists at least one sonic density $R_s\in(R_-,R_+)$ satisfying \Cref{eq:sonic-density}.
In \Cref{appendix:barotropic_eos,appendix:eos-alt-criterion}, we provide mild assumptions on the physical pressure $\tilde{p}(\rho,e)$ to ensure that the reduced pressure is convex, $p''(R) \geq 0$ for all $R >0$. Assuming this is the case, one has
\begin{equation} \label{eq:Sdoubleprime}
\mathcal{S}''(R)=\frac{d}{dR}\Bigl(\frac{m^2}{R^2}-p'(R)\Bigr)
=-\frac{2m^2}{R^3}-p''(R)<0,
\end{equation}
so $\mathcal{S}'$ is strictly decreasing on $(0,\infty)$ and the sonic density in \Cref{eq:sonic-density} is unique. Thus, $\mathcal{S}$ is strictly concave, and $R_s$ is the unique global maximizer of $\mathcal{S}$.

The sonic density plays a central role in the structure of IGR shock profiles. As proved later in \Cref{sec:global}, the density derivative $\dot R$ diverges at the sonic state, but this singularity is sufficiently mild that the profile remains H\"{o}lder continuous there. 
This will allow us to quantify the regularity of the profile across the IGR shock despite the loss of classical smoothness at the sonic crossing. Thus, the inclusion of the entropic pressure in the IGR model yields enhanced regularity relative to a purely discontinuous shock description, even though the profile must still be interpreted in the weak sense because classical differentiability fails at a single point.

%-------------------------------------------------------------------------------------------
\subsubsection{IGR shock admissibility criteria}\label{sec:igr-admissibility}

Taking $\xi\to\pm\infty$ in the integrated relations \Cref{eq:fluxes} and using $S(\pm\infty)=0$ yields the Rankine--Hugoniot conditions \cite{courant1948supersonic}
\begin{equation}\label{eq:RH}
    R_\pm(U_\pm-c)=m,\qquad
    mU_\pm+p(R_\pm)=q,
\end{equation}
together with the energy-flux relation fixing $k$ (equivalently, fixing the TW energy curve \Cref{eq:e-explicit}). In particular, if $(R_\pm,U_\pm)$ satisfy \Cref{eq:RH}, then the wave speed and the mass and momentum fluxes are uniquely determined by
\begin{subequations}\label{eq:wavespeed-and-fluxes}
\begin{align}
    c&=\frac{R_+U_+-R_-U_-}{R_+-R_-},\qquad
    m=R_-(U_--c)=R_+(U_+-c),\\
    q&=mU_-+p(R_-)=mU_+ + p(R_+),
\end{align}
\end{subequations}
and the far-field conditions $S(\pm\infty)=0$ are equivalently $\mathcal{S}(R_\pm)=0$.

We focus on \emph{compressive} (shock-like) connections and label the end states so that $R_-<R_+$.
Equivalently, with \(m=R_\pm(U_\pm-c)>0\), this ordering is
the same as \(U_->U_+\).
For the compressive branch studied here, we impose monotone increasing density profiles as the admissible class (see \cite{riemann_fluid_flow} for a detailed discussion of monotonicity for shocks). This is formalized in \Cref{sec:mono-Z}, where strict monotonicity is then proved for nonconstant admissible profiles.

%-------------------------------------------------------------------------------------------
\subsection{Standing assumptions and main theorem}\label{sec:standing-assumptions}

\begin{assumption}\label{assump:main}
Unless stated otherwise, we assume throughout:
\begin{itemize}
\item $R_\pm>0$ with $R_-\neq R_+$, and $(R_\pm,U_\pm)$ satisfy the Rankine--Hugoniot relations \Cref{eq:RH} with $S(\pm\infty)=0$ (so that $c,m,q$ are given by \Cref{eq:wavespeed-and-fluxes}, and $k$ is fixed by the energy-flux condition);
\item compressive ordering: $R_-<R_+$;
\item the end states satisfy the $1$-shock Lax inequalities for the reduced barotropic Euler subsystem, $U_- - \sqrt{p'(R_-)} > c > U_+ - \sqrt{p'(R_+)}$,
which is equivalent to the transonic sign change \Cref{eq:transonic-sign} under the orientation $m>0$ implied by the left Lax inequality;
\item the reduced pressure $p\in C^2((0,\infty))$ satisfies one of the convexity admissibility criteria in \Cref{appendix:barotropic_eos} or \Cref{appendix:eos-alt-criterion} (in particular $p''(R)\ge 0$ for $R>0$), and $p''$ is locally Lipschitz in a neighborhood of the unique sonic density $R_s$; equivalently, $p\in C^{2,1}$ near $R_s$.
\end{itemize}
\end{assumption}
Note that these assumptions imply $m\neq0$, so that we explicitly disqualify constant profiles. Admissible profiles are selected by a monotonicity principle introduced in \Cref{sec:mono-Z}; in the present setting this selection implies strict monotonicity. We now state a summary version of the main result.

\begin{theorem}[Existence/uniqueness/regularity of IGR shock profiles]\label{thm:main}
Under the standing assumptions above, there exists a unique (modulo translation) admissible global transonic traveling-wave profile $R$
connecting $R_-$ to $R_+$. This profile is strictly increasing, continuous across the sonic crossing, and
twice continuously differentiable away from the crossing. It uniquely determines
$U$, $e$, $S$, and $\mathcal E$, and together $(R,U,\mathcal E,S)$ solve the traveling-wave system
\Crefrange{eq:TWa}{eq:TWd} globally in the weak sense, with $e$ recovered from the TW energy curve
\Cref{eq:e-explicit}. 
\end{theorem}

The proof, including the H\"older, Sobolev, tail-decay, and away-from-degeneracy
higher-regularity statements, is given in \Cref{sec:global}, especially
\Cref{thm:global-glue}, \Cref{thm:unique-translation}, and
\Cref{thm:full-solution}.

%============================================================================================
\section{Structure of the IGR shock profile}\label{sec:structure-shock}

Before establishing global existence and regularity in \Cref{sec:global}, we record the basic phase-portrait structure of the traveling-wave equation. We identify the nonsonic equilibria associated with the end states, prove strict monotonicity and nonsonic regularity for admissible weak profiles, and then reduce the profile equation to a linear first-order equation for the squared slope $Z=\dot R^{\,2}$. This yields canonical one-sided solution formulas and the reconstruction of one-sided profile branches.

%-------------------------------------------------------------------------------------------
\subsection{Equilibria and asymptotic states}

We first build intuition about the solution profile of the IGR shock equation by looking at its 
phase-portrait character. We find that the asymptotic values of the density profile are 
hyperbolic fixed points, and that we can formally interpret the traveling-wave profile 
as a heteroclinic orbit connecting these hyperbolic saddle points.

\paragraph{Planar form away from sonic states}
Expanding the derivative in \Cref{eq:R-ODE-general} gives
\begin{equation}\label{eq:R-ODE-expanded}
    \frac{\mathcal{S}(R)}{R}
    -\alpha\frac{\mathcal{S}'(R)}{R}\ddot{R}
    -\alpha\left(\frac{\mathcal{S}''(R)}{R}-\frac{\mathcal{S}'(R)}{R^2}+\frac{2m^2}{R^4}\right)\dot{R}^2
    =0.
\end{equation}
On intervals where $\mathcal{S}'(R)\neq 0$, i.e., $R \neq R_s$, this is equivalent to
\begin{equation}\label{eq:planar}
\begin{cases}
\dot{R}=V,\\[0.5ex]
\dot{V}=\displaystyle
\frac{\mathcal{S}(R)}{\alpha\,\mathcal{S}'(R)}
-\left(
\frac{\mathcal{S}''(R)}{\mathcal{S}'(R)}
-\frac{1}{R}
+\frac{2m^2}{R^3\mathcal{S}'(R)}
\right)V^2.
\end{cases}
\end{equation}

\begin{lemma}[Nonsonic equilibria and linearization]\label{lem:equilibria}
Assume $R_\ast>0$ satisfies $\mathcal{S}(R_\ast)=0$ and $\mathcal{S}'(R_\ast)\neq 0$.
Then $(R_\ast,0)$ is an equilibrium of \Cref{eq:planar}, and the Jacobian at this point is
\[
J(R_\ast,0)=
\begin{pmatrix}
0 & 1\\[0.5ex]
\frac{1}{\alpha} & 0
\end{pmatrix},
\]
with eigenvalues $\lambda_\pm=\pm \alpha^{-1/2}$. In particular, $(R_\ast,0)$ is a hyperbolic saddle point (see, e.g., \cite{PETERS2000287, smoller1994shock}).
\end{lemma}

\begin{proof}
At $V=0$, the second equation in \Cref{eq:planar} reduces to
$\dot{V}=\mathcal{S}(R)/(\alpha \mathcal{S}'(R))$, so equilibria satisfy $V=0$ and $\mathcal{S}(R)=0$.
Linearizing at $(R_\ast,0)$ gives
\[
\partial_R(\dot{R})=0,\quad \partial_V(\dot{R})=1,\quad \partial_V(\dot{V})|_{V=0}=0,\quad
\partial_R(\dot{V})|_{V=0}=\frac{1}{\alpha}\frac{d}{dR}\!\left(\frac{\mathcal{S}}{\mathcal{S}'}\right)_{R=R_\ast}.
\]
Since $\mathcal{S}(R_\ast)=0$ and $\mathcal{S}'(R_\ast)\neq 0$,
$\frac{d}{dR}\!\left(\frac{\mathcal{S}}{\mathcal{S}'}\right)_{R=R_\ast}
=\frac{(\mathcal{S}')^2-\mathcal{S}\mathcal{S}''}{(\mathcal{S}')^2}\Big|_{R=R_\ast}
=1$.
Hence the stated Jacobian and eigenvalues follow.
\end{proof}

\begin{remark}[Asymptotic states]
Because $S(\pm\infty)=0$ and $S(\xi)=\mathcal{S}(R(\xi))$, one has
$\mathcal{S}(R_\pm)=0$, implying $(R_\pm,0)$ are equilibria of \Cref{eq:planar}. Under \Cref{eq:sonic-density,eq:transonic-sign}, $\mathcal{S}'(R_\pm)\neq 0$, so
$(R_\pm,0)$ are nonsonic hyperbolic saddle points of \Cref{eq:planar}. 
Thus, at the level of the nonsonic phase portrait, a sufficiently regular
traveling-wave shock should be viewed as a heteroclinic connection between the
saddle equilibria $(R_-,0)$ and $(R_+,0)$.
\end{remark}

%-------------------------------------------------------------------------------------------
\subsection{Monotonicity of the IGR shock profile}\label{sec:mono-Z}

We now formulate weak solutions of the IGR shock profile equation and impose admissibility by restricting to the monotone branch connecting the two end states. This choice is motivated by the heteroclinic structure of the shock profiles, is consistent with the sharp-shock limit $\alpha\to 0$, and excludes oscillatory boundary-layer behavior. We then show that every nonconstant admissible weak profile is strictly monotone and restricts to a classical solution away from the sonic density.

\begin{definition}[Weak IGR shock profile]\label{def:weak-profile}
A weak IGR shock profile is a function $R:\mathbb{R}\to [R_-,R_+]$ such that:
\begin{enumerate}[label=(\roman*)]
\item $R\in W^{1,2}_{\mathrm{loc}}(\mathbb{R})$;
\item $R$ connects the prescribed end states, i.e.,
$\lim_{\xi\to-\infty}R(\xi)=R_-$ and
$\lim_{\xi\to+\infty}R(\xi)=R_+$.
\item $R$ satisfies the IGR profile equation \Cref{eq:R-ODE-general} in the sense of distributions: for every
$\varphi\in C_c^\infty(\mathbb{R})$,
\begin{equation}\label{eq:weak-profile}
\int_{\mathbb{R}}
\alpha\,\frac{\mathcal{S}'(R)}{R}\,\dot{R}\,\dot{\varphi} \,d\xi
+
\int_{\mathbb{R}}
\left(
\frac{\mathcal{S}(R)}{R}
-2\alpha\,\frac{m^2}{R^4}\dot{R}^2
\right)\varphi\,d\xi
=0.
\end{equation}
\end{enumerate}
\end{definition}

\begin{definition}[Admissible monotone profile class]\label{def:M}
The admissible monotone profile class $\mathcal{M}$ consists of all weak IGR shock profiles
$R:\mathbb{R}\to [R_-,R_+]$ that are nondecreasing on $\mathbb{R}$.
\end{definition}

Before proving further results about the IGR shock profile, we first record the root and concavity structure of $\mathcal{S}$.

\begin{lemma}[Root structure of $\mathcal{S}$]\label{lem:S-two-roots}
Under the standing assumptions (\Cref{assump:main}), one has:
\begin{enumerate}[label=\textup{(\roman*)}]
\item $\mathcal{S}''(R)<0$ for all $R>0$. In particular, there is a unique
$R_s\in(R_-,R_+)$ such that $\mathcal{S}'(R_s)=0$, $\mathcal{S}'(R)>0$ for $R\in[R_-,R_s)$, and
$\mathcal{S}'(R)<0$ for $R\in(R_s,R_+]$.

\item $\mathcal{S}(R)>0$ for all $R\in(R_-,R_+)$.

\item The only zeros of $\mathcal{S}$ in $(0,\infty)$ are $R_-$ and $R_+$.
\end{enumerate}
\end{lemma}

\begin{proof}
Recall that
\[
    \mathcal{S}(R) =q-mc-\frac{m^2}{R}-p(R),
    \quad 
    \mathcal{S}'(R) = \frac{m^2}{R^2}-p'(R) \,,
    \quad \text{and} \quad
    \mathcal{S}''(R)=-\frac{2m^2}{R^3}-p''(R) \,.
\]
By \Cref{eq:RH} and since $S(\pm\infty)=0$, we have $\mathcal{S}(R_\pm)=0$. Using $m\neq 0$ and $p''(R)\ge 0$ from \Cref{assump:main}, we obtain $\mathcal{S}''(R)<0$ for all $R>0$. Hence, $\mathcal{S}'$ is strictly decreasing on $(0,\infty)$.

Since \Cref{sec:degeneracy-sonic-density} already established the transonic sign change
\Cref{eq:transonic-sign}, there is a unique $R_s\in(R_-,R_+)$ such that $\mathcal{S}'(R_s)=0$, with
\[
\mathcal{S}'(R)>0 \quad \text{for }R\in(R_-,R_s),
\qquad
\mathcal{S}'(R)<0 \quad \text{for }R\in(R_s,R_+).
\]
Since $\mathcal{S}'$ is continuous and positive on $(R_-,R_s)$, we also have $\mathcal{S}'(R_-)\ge 0$. If $\mathcal{S}'(R_-)=0$, then strict decrease of $\mathcal{S}'$ implies
$\mathcal{S}'(R)<0$ for every $R>R_-$, contradicting positivity on $(R_-,R_s)$.
Hence, $\mathcal{S}'(R_-)>0$. An analogous argument gives $\mathcal{S}'(R_+)<0$. This proves \textup{(i)}.

Since $\mathcal{S}$ is strictly concave and $\mathcal{S}(R_\pm)=0$, it follows that
$\mathcal{S}(R)>0$ for all $R\in(R_-,R_+)$,
proving \textup{(ii)}. Finally, a strictly concave function has at most two zeros, so
$R_-$ and $R_+$ are the only zeros of $\mathcal{S}$ in $(0,\infty)$, proving \textup{(iii)}.
\end{proof}

The weak formulation in \Cref{def:weak-profile} naturally singles out the following
flux-like slope variable, 
$W:=\frac{\mathcal{S}'(R)}{R}\dot{R}$.
This is precisely the quantity appearing under the derivative in
\Cref{eq:R-ODE-general}. On any open interval where $\mathcal{S}'(R)\neq 0$,
the weak profile equation can be written as the first-order system
\begin{equation}\label{eq:RW-system}
\begin{cases}
\dot{R}=\dfrac{R}{\mathcal{S}'(R)}\,W,\\[1ex]
\dot{W}=\dfrac{\mathcal{S}(R)}{\alpha R}
-\dfrac{2m^2}{R^2\mathcal{S}'(R)^2}\,W^2.
\end{cases}
\end{equation}
Initially, this system is understood in the sense of
distributions; after the local regularity argument below, the system holds
classically on nonsonic intervals. Under \Cref{assump:main}, the vector field in
\Cref{eq:RW-system} is $C^1$ on compact subsets of the nonsonic region, where
$\mathcal{S}'$ is bounded away from zero. Hence the corresponding classical
initial-value problem enjoys local uniqueness on nonsonic intervals.

\begin{theorem}[Strict monotonicity and nonsonic regularity]\label{thm:strict-mono}
Let $R\in\mathcal{M}$ be a nonconstant admissible weak solution of the IGR shock profile equation. Then the following holds:
\begin{enumerate}[label=(\roman*)]
\item $R$ is strictly increasing on $\mathbb{R}$; in particular,
$R(\xi)\in(R_-,R_+)$ for all $\xi\in\mathbb{R}$.
\item Let $\Omega:= \{\xi\in\mathbb{R}:\ R(\xi)\neq R_s\}$.
Then $R$ restricts to a $C^2$ function on $\Omega$, and
$\dot{R}(\xi)>0$ for all $\xi\in\Omega$.
\end{enumerate}
\end{theorem}

\begin{proof}
\emph{Step 1: local regularity on nonsonic neighborhoods.}
A bootstrapping argument allows us to establish regularity on open intervals away from the sonic point. 
Throughout the proof we work with the continuous or absolutely continuous
representatives supplied by the one-dimensional Sobolev embedding theorem. In
particular, on bounded intervals, functions in $W^{1,2}$ have continuous
representatives, and functions in $W^{1,1}$ have absolutely continuous
representatives.

Set $a(\xi):=\frac{\mathcal S'(R(\xi))}{R(\xi)}$, 
and let $B\subset\mathbb{R}$ be a bounded open interval on which $|\mathcal S'(R(\xi))|\ge c_0>0$.
Since $R\in W^{1,2}(B)$ and $B$ is bounded, $R$ has a continuous representative on $B$.
Also, because $R$ is nondecreasing with
$\lim_{\xi\to-\infty}R(\xi)=R_-$ and
$\lim_{\xi\to+\infty}R(\xi)=R_+$,
we have $R(\xi)\in [R_-,R_+] \subset (0,\infty)$
for all $\xi\in\mathbb R$.
Hence, $R(B)$ is contained in a compact subset of $(0,\infty)$. Since
$|\mathcal S'(R)|\ge c_0$ on $B$, both
$a=\mathcal{S}'(R)/R$ and
$a^{-1}$
are bounded on $B$.

Because $a\in L^\infty(B)$ and $\dot R\in L^2(B)$, we have $W=a\,\dot R\in L^2(B)\subset L^1(B)$.
Restricting \Cref{eq:weak-profile} to test functions $\varphi\in C_c^\infty(B)$ gives
\[
\int_B \alpha W\,\dot\varphi\,d\xi+\int_B G\,\varphi\,d\xi=0,
\quad \text{where} \quad
G:=\frac{\mathcal S(R)}{R}-2\alpha\,\frac{m^2}{R^4}\dot R^{\,2}.
\]
Since $R$ stays in a compact subset of $(0,\infty)$ on $B$ and $\dot R\in L^2(B)$, we have
$G\in L^1(B)$. Therefore, the weak derivative of $W$ is given by $\dot W=G/\alpha$, so that $W\in W^{1,1}(B)$.
Thus $W$ has an absolutely continuous representative on $B$, and in particular
$W\in C^0(B)$.
Since $a^{-1}$ is continuous on $B$, it follows that $\dot R=a^{-1}W\in C^0(B)$,
and hence $R\in C^1(B)$.

Now that $R\in C^1(B)$, the function $G$
is continuous on $B$. Since $W$ is absolutely continuous and its weak derivative is
$\dot W=G/\alpha\in C^0(B)$, we have $W\in C^1(B)$.
Moreover, because $R\in C^1(B)$ and $|\mathcal S'(R)|\ge c_0$ on $B$,
the coefficient $a^{-1}=R/\mathcal S'(R)$ belongs to $C^1(B)$. Thus
$\dot R=a^{-1}W\in C^1(B)$,
so that $R\in C^2(B)$. Consequently, the identities in \Cref{eq:RW-system} hold pointwise on $B$.

\emph{Step 2: end states are not reached at finite $\xi$.}
Suppose $R(\xi_0)=R_-$ for some finite $\xi_0$. Since $R$ is nondecreasing and
$\lim_{\xi\to-\infty}R(\xi)=R_-$,
we have $R\equiv R_-$ on $(-\infty,\xi_0]$, hence $\dot R=0$ a.e. there. Defining $\xi_\ast:=\sup\{\xi:\,R(\xi)=R_-\}$, we have $\xi_\ast<\infty$ since $R$ is nonconstant and
$\lim_{\xi\to+\infty}R(\xi)=R_+>R_-$.
Also, $R(\xi_\ast)=R_-$ by continuity. Choosing a small nonsonic neighborhood $B_{\xi_\ast}\ni \xi_\ast$
and applying Step 1 again, we obtain $R\in C^2(B_{\xi_\ast})$ and $W\in C^1(B_{\xi_\ast})$, so
$(R,W)$ solves \Cref{eq:RW-system} classically on $B_{\xi_\ast}$.

Since $W\in C^0(B_{\xi_\ast})$ and $W=0$ a.e. on
$B_{\xi_\ast}\cap(-\infty,\xi_\ast]$, it follows that
$W=0$ on $B_{\xi_\ast}\cap(-\infty,\xi_\ast]$. In particular,
$W(\xi_\ast)=0$.
Therefore, $(R(\xi_\ast),W(\xi_\ast))=(R_-,0)$. Since $\mathcal S(R_-)=0$, the
right-hand side of \Cref{eq:RW-system} vanishes at $(R_-,0)$, so $(R_-,0)$
defines a constant solution of the first-order system.
On the
nonsonic neighborhood $B_{\xi_\ast}$, the right-hand side of \Cref{eq:RW-system} is $C^1$, so local
uniqueness implies $R\equiv R_-$ in a neighborhood of $\xi_\ast$, 
contradicting the definition of $\xi_\ast$ as the
rightmost contact point. Hence $R(\xi)>R_-$ for all $\xi\in\mathbb{R}$. 
The argument for $R_+$ is symmetric. Therefore
$R(\xi)\in (R_-,R_+)$ for all $\xi\in\mathbb{R}$.

\emph{Step 3: strict monotonicity.}
If $R(\xi_1)=R(\xi_2)$ for some $\xi_1<\xi_2$, monotonicity implies
$R\equiv R_0$ on $[\xi_1,\xi_2]$. Testing \Cref{eq:weak-profile} with
$\varphi\in C_c^\infty((\xi_1,\xi_2))$ gives $\mathcal{S}(R_0)=0$.
By \Cref{lem:S-two-roots}(iii), $R_0\in\{R_-,R_+\}$, contradicting Step 2.
Hence $R$ is strictly increasing, proving (i).

\emph{Step 4: nonsonic regularity.}
Let $\Omega:=\{\xi\in\mathbb{R}:\,R(\xi)\neq R_s\}$.
Fix $\xi_0\in\Omega$. Then on some neighborhood $B_{\xi_0}\ni\xi_0$ one has
$|\mathcal{S}'(R(\xi))|\ge c_0>0$.
By Step 1, it follows that $R\in C^2(B_{\xi_0})$. Since $\xi_0$ was arbitrary, $R$ is $C^2$ on
$\Omega$.

\emph{Step 5: positivity of $\dot R$ on $\Omega$.}
Fix $\xi_0\in\Omega$ and choose $B_{\xi_0}$ as in Step 4. Since $R$ is nondecreasing and
$R\in C^2(B_{\xi_0})$, we have $\dot R\ge 0$ on $B_{\xi_0}$.
Suppose $\dot R(\xi_\ast)=0$ for some $\xi_\ast\in B_{\xi_0}$. Then
$\xi_\ast$ is a local minimum of $\dot R$, and hence $\ddot R(\xi_\ast)=0$.
Thus $W(\xi_\ast)=0$ and $\dot W(\xi_\ast)=0$. But evaluating
\Cref{eq:RW-system} at $\xi_\ast$ gives
\[
\dot W(\xi_\ast)=\frac{\mathcal S(R(\xi_\ast))}{\alpha R(\xi_\ast)}>0,
\]
because $R(\xi_\ast)\in(R_-,R_+)$ and \Cref{lem:S-two-roots}(ii) applies. This
contradiction proves $\dot R>0$ on $B_{\xi_0}$. Since $\xi_0$ was arbitrary,
$\dot R>0$ on $\Omega$. This proves \textup{(ii)}.
\end{proof}

\begin{remark}[Higher regularity away from the sonic point]\label{rem:nonsonic-higher-regularity}
Under \Cref{assump:main}, the nonsonic regularity obtained above is $C^2$. If, on a given nonsonic
density interval, the reduced pressure $p$ has additional regularity, then one obtains 
a corresponding higher regularity for $R$ on the associated $\xi$-interval.
\end{remark}

\begin{remark}[Unique sonic crossing]\label{rem:unique-sonic-crossing}
By continuity and strict monotonicity, every nonconstant $R\in\mathcal{M}$ crosses the unique sonic
density $R_s$ at exactly one point $\xi_s\in\mathbb{R}$, characterized by $R(\xi_s)=R_s$.
\end{remark}

\begin{lemma}[Flattening of the density tails]\label{lem:tail-flattening}
Let $R\in\mathcal{M}$ be a nonconstant admissible weak solution of the IGR shock
profile equation. Then
\[
\lim_{\xi\to-\infty}\dot R(\xi)=0=
\lim_{\xi\to+\infty}\dot R(\xi).
\]
\end{lemma}

\begin{proof}
By \Cref{thm:strict-mono} and \Cref{rem:unique-sonic-crossing}, the profile is
$C^2$ on each tail $(-\infty,\xi_s)$ and $(\xi_s,\infty)$, with
$\dot R>0$ there. Hence the first-order system \Cref{eq:RW-system} holds
classically on both tails. We write it as
\[
\dot R=\frac{R}{\mathcal S'(R)}W,
\qquad
\dot W=A(R)-B(R)W^2,
\]
where $A(R):=\frac{\mathcal S(R)}{\alpha R}$ and
$B(R):=\frac{2m^2}{R^2\mathcal S'(R)^2}$.
By \Cref{lem:S-two-roots}, $\mathcal S(R_\pm)=0$ and
$\mathcal S'(R_\pm)\neq 0$. Therefore
\[
A(R(\xi))\to0
\quad\text{as }\xi\to\pm\infty,
\]
while $B(R(\xi))$ remains bounded below by a positive constant near each
endpoint.

We first prove the claim as $\xi\to-\infty$. On the left tail,
$R(\xi)\in(R_-,R_s)$, so $\mathcal S'(R)>0$. Since $\dot R>0$, it follows that
$W>0$ on $(-\infty,\xi_s)$. Fix $\varepsilon>0$. Choosing
$\xi_\varepsilon<\xi_s$ sufficiently negative, there exists $b_\varepsilon>0$
such that, for all $\xi\le \xi_\varepsilon$,
\[
0\le A(R(\xi))\le \frac{b_\varepsilon}{2}\varepsilon^2,
\qquad
B(R(\xi))\ge b_\varepsilon .
\]
Suppose, for contradiction, that $W(\xi_0)>\varepsilon$ for some
$\xi_0\le\xi_\varepsilon$. Define $Y(\tau):=W(\xi_0-\tau)$ where $\tau\ge0$.
Since $\xi_0-\tau\le \xi_\varepsilon<\xi_s$ for every $\tau\ge0$, the trajectory
remains on the left nonsonic tail, and the bounds on $A$ and $B$ hold for all
$\tau\ge0$.

Let $c:=b_\varepsilon/2$. As long as $Y(\tau)\ge\varepsilon$, we have
\[
Y'(\tau)
=
-\dot W(\xi_0-\tau)
=
-A(R(\xi_0-\tau))+B(R(\xi_0-\tau))Y(\tau)^2
\ge
c\,Y(\tau)^2.
\]
Since $Y(0)>\varepsilon$, this inequality prevents $Y$ from crossing below
$\varepsilon$: indeed, while $Y\ge\varepsilon$, one has $Y'>0$. Hence the
inequality holds for all $\tau\ge0$ for which the solution exists. Therefore
\[
\frac{d}{d\tau}\frac{1}{Y(\tau)}
=
-\frac{Y'(\tau)}{Y(\tau)^2}
\le
-c.
\]
Integrating gives
\[
\frac{1}{Y(\tau)}
\le
\frac{1}{Y(0)}-c\tau.
\]
For $\tau>1/(cY(0))$, the right-hand side is negative, which is impossible
because $Y(\tau)>0$. This contradicts the fact that $Y(\tau)=W(\xi_0-\tau)$ is
defined and finite for every finite $\tau\ge0$ on the left nonsonic tail.
Therefore $0<W(\xi)\le\varepsilon$ for all sufficiently negative
$\xi$. Since $\varepsilon>0$ was arbitrary,
$W(\xi)\to0$ as $\xi\to-\infty$.
The coefficient $R/\mathcal S'(R)$ remains bounded near $R_-$, and therefore
\[
\dot R(\xi)=\frac{R(\xi)}{\mathcal S'(R(\xi))}W(\xi)\to0
\qquad\text{as }\xi\to-\infty.
\]

The proof as $\xi\to+\infty$ is analogous. On the right tail,
$\mathcal S'(R)<0$, so $W<0$; applying the same argument to $-W$ forward in
$\xi$ shows that $W(\xi)\to0$ as $\xi\to+\infty$. Since
$R/\mathcal S'(R)$ remains bounded near $R_+$, it follows that
$\dot R(\xi)\to0$ as $\xi\to+\infty$.
\end{proof}

%-------------------------------------------------------------------------------------------
\subsection{First-order IGR shock equation and one-sided reconstruction}\label{sec:one-sided-reconstruct}

Having established strict monotonicity, the density itself can be used as the independent variable on each nonsonic branch.

\begin{proposition}[Reduction to a linear first-order equation]\label{prop:Z-equation}
Let $R\in\mathcal{M}$ be nonconstant, and let $\xi_s$ be the unique sonic crossing from
\Cref{rem:unique-sonic-crossing}. By \Cref{thm:strict-mono}, the restrictions of $R$
to $(-\infty,\xi_s)$ and $(\xi_s,\infty)$ are $C^2$ and strictly increasing, with images
$(R_-,R_s)$ and $(R_s,R_+)$, respectively. Define
\[
Z(R):=\dot{R}(\xi)^2,\qquad R=R(\xi).
\]
This is well-defined on each of the density intervals $(R_-,R_s)$ and $(R_s,R_+)$.
Then, on each interval,
\begin{equation}\label{eq:Z-general}
Z' + P(R)\,Z = Q(R),
\end{equation}
where
\begin{equation}\label{eq:PQ}
P(R):=2\left(
\frac{\mathcal{S}''(R)}{\mathcal{S}'(R)}
-\frac{1}{R}
+\frac{2m^2}{R^3\mathcal{S}'(R)}
\right)
=
-\frac{2}{R}-\frac{2p''(R)}{\mathcal{S}'(R)},
\qquad
Q(R):=\frac{2}{\alpha}\frac{\mathcal{S}(R)}{\mathcal{S}'(R)}.
\end{equation}
Moreover, $Z$ extends continuously to the outer endpoints and satisfies
\begin{equation}\label{eq:Z-end-bc}
Z(R_-)=0,\qquad Z(R_+)=0.
\end{equation}
Finally, $Z(R)>0$ for $R\in(R_-,R_+)\setminus\{R_s\}$.
\end{proposition}
\begin{proof}
On each of the intervals $(-\infty,\xi_s)$ and $(\xi_s,\infty)$,
\Cref{thm:strict-mono} gives $R\in C^2$ and $\dot R>0$. Hence $R$ may be used
as the independent variable there. Differentiating the definition of $Z$ and the chain rule yields
\[
\ddot R
=
\frac{d\dot R}{dR}\dot R
=
\frac12 Z'(R).
\]
Substituting this identity into \Cref{eq:R-ODE-expanded} and dividing by
$\alpha\,\mathcal{S}'(R)/R$ gives \Cref{eq:Z-general,eq:PQ}.

The endpoint conditions follow from \Cref{lem:tail-flattening}. Indeed, as
$\xi\to\pm\infty$ we have $R(\xi)\to R_\pm$ by the definition of a weak IGR shock
profile, while \Cref{lem:tail-flattening} gives $\dot R(\xi)\to0$. Therefore the
corresponding one-sided squared-slope functions extend continuously to the
outer endpoints with $Z(R_-)=0$ and $Z(R_+)=0$.
Finally, \Cref{thm:strict-mono} gives $\dot R>0$ on
$(-\infty,\xi_s)\cup(\xi_s,\infty)$, and hence,
$Z(R)>0$ for $R\in(R_-,R_+)\setminus\{R_s\}$.
\end{proof}

\begin{lemma}[Canonical one-sided solution formulas]\label{lem:Z-solution-formula}
Let $P$ and $Q$ be the coefficients in \Cref{eq:PQ}. Fix reference points
$R_-^\ast\in(R_-,R_s)$ and $R_+^\ast\in(R_s,R_+)$, and define
\[
\mu_-(R):=\exp\!\left(\int_{R_-^\ast}^{R}P(s)\,ds\right),
\qquad R\in(R_-,R_s),
\]
and
\[
\mu_+(R):=\exp\!\left(\int_{R_+^\ast}^{R}P(s)\,ds\right),
\qquad R\in(R_s,R_+).
\]
Then there is a unique solution
\[
Z_-\in C^1((R_-,R_s))\cap C^0([R_-,R_s))
\]
of \eqref{eq:Z-general} on $(R_-,R_s)$ with $Z_-(R_-)=0$, where the endpoint value is understood
by continuous extension. It is given by
\begin{equation}\label{eq:Zminus-formula}
Z_-(R)
=
\mu_-(R)^{-1}
\int_{R_-}^{R}\mu_-(s)Q(s)\,ds,
\qquad R\in(R_-,R_s).
\end{equation}
Similarly, there is a unique solution
\[
Z_+\in C^1((R_s,R_+))\cap C^0((R_s,R_+])
\]
of \eqref{eq:Z-general} on $(R_s,R_+)$ with $Z_+(R_+)=0$, given by
\begin{equation}\label{eq:Zplus-formula}
Z_+(R)
=
\mu_+(R)^{-1}
\int_{R_+}^{R}\mu_+(s)Q(s)\,ds,
\qquad R\in(R_s,R_+).
\end{equation}
Both formulas are independent of the choice of reference points
$R_-^\ast$ and $R_+^\ast$. Moreover,
\[
Z_-(R)>0\quad\text{for }R\in(R_-,R_s),
\qquad
Z_+(R)>0\quad\text{for }R\in(R_s,R_+).
\]
\end{lemma}

\begin{proof}
Since $\mathcal S'(R_-)\neq 0$ and $\mathcal S'(R_+)\neq 0$, the coefficients
$P$ and $Q$ extend continuously to the corresponding outer endpoints. Thus
$\mu_-$ and $\mu_+$ also admit positive continuous extensions to those endpoints,
and the integrals in \Cref{eq:Zminus-formula,eq:Zplus-formula} are well defined.

Multiplying the equation \eqref{eq:Z-general} by the integrating factor $\mu_\pm$ gives
$(\mu_\pm Z)'=\mu_\pm Q$.
Integrating from $R_-$ to $R$ on the left branch, and from $R_+$ to $R$ on the
right branch, yields \Cref{eq:Zminus-formula,eq:Zplus-formula}. Differentiating
these formulas shows that they solve the equation on their respective open
intervals, and the continuity of $P$, $Q$, and $\mu_\pm$ at the outer endpoints
shows that the prescribed endpoint values are attained by continuous extension.

For uniqueness, let $\widetilde Z$ be another solution on the left branch with
$\widetilde Z(R_-)=0$ by continuous extension. Then the difference
$Y:=Z_- - \widetilde Z$ satisfies
$(\mu_-Y)'=0$.
on $(R_-,R_s)$. Hence $\mu_-Y$ is constant there. Taking the limit
$R\downarrow R_-$ gives this constant equal to zero, so $Y\equiv0$. The right
branch is identical.

It remains to prove positivity. By \Cref{lem:S-two-roots},
$\mathcal S(R)>0$ on $(R_-,R_+)$,
while
\[
\mathcal S'(R)>0 \quad\text{on }(R_-,R_s),
\qquad
\mathcal S'(R)<0 \quad\text{on }(R_s,R_+).
\]
Therefore $Q(R)$
is positive on $(R_-,R_s)$ and negative on $(R_s,R_+)$. Since
$\mu_\pm>0$, the formula \Cref{eq:Zminus-formula} immediately gives
$Z_-(R)>0$ for $R\in(R_-,R_s)$. On the right branch, for
$R\in(R_s,R_+)$,
\[
\int_{R_+}^{R}\mu_+(s)Q(s)\,ds
=
-\int_{R}^{R_+}\mu_+(s)Q(s)\,ds>0,
\]
because $\mu_+Q<0$ on $(R_s,R_+)$. Hence \Cref{eq:Zplus-formula} gives
$Z_+(R)>0$ for $R\in(R_s,R_+)$.

Finally, changing the reference point in the definition of $\mu_\pm$ only
multiplies $\mu_\pm$ by a positive constant, which cancels from the formulas
for $Z_\pm$.
\end{proof}

\begin{corollary}[Identification of admissible branches with $Z_\pm$]\label{cor:Z-identification}
Let $R\in\mathcal{M}$ be nonconstant, and let $Z(R)=\dot R(\xi)^2$ be the function from
\Cref{prop:Z-equation}. Then
\[
Z(R)=Z_-(R)\quad\text{on }(R_-,R_s),
\qquad
Z(R)=Z_+(R)\quad\text{on }(R_s,R_+),
\]
where $Z_\pm$ are the canonical one-sided solutions from
\Cref{lem:Z-solution-formula}. 
In particular, the squared slope of every admissible nonconstant profile is
uniquely determined on each nonsonic branch. The corresponding profile branches
are then recovered by quadrature, up to translation.
\end{corollary}

\begin{proof}
By \Cref{prop:Z-equation}, the function $Z$ solves \Cref{eq:Z-general} on
$(R_-,R_s)$ and $(R_s,R_+)$ with
$Z(R_-)=0$ and $Z(R_+)=0$.
\Cref{lem:Z-solution-formula} gives uniqueness of the corresponding one-sided solutions.
Therefore $Z=Z_-$ on $(R_-,R_s)$ and $Z=Z_+$ on $(R_s,R_+)$.
\end{proof}

\begin{lemma}[Outer-endpoint expansion of the one-sided solutions]\label{lem:Z-endpoints}
The one-sided solutions $Z_-$ and $Z_+$ from \Cref{lem:Z-solution-formula} satisfy
\[
Z_-(R)=\frac{1}{\alpha}(R-R_-)^2+O\!\left(|R-R_-|^3\right)
\qquad (R\downarrow R_-),
\]
and
\[
Z_+(R)=\frac{1}{\alpha}(R-R_+)^2+O\!\left(|R-R_+|^3\right)
\qquad (R\uparrow R_+).
\]
Moreover, for any fixed
$R_0^-\in(R_-,R_s)$ and $R_0^+\in(R_s,R_+)$,
\[
\int_{R}^{R_0^-}\frac{d\rho}{\sqrt{Z_-(\rho)}}
=
\sqrt{\alpha}\,\log\!\frac{1}{R-R_-}
+ O(1)
\qquad (R\downarrow R_-),
\]
and
\[
\int_{R_0^+}^{R}\frac{d\rho}{\sqrt{Z_+(\rho)}}
=
\sqrt{\alpha}\,\log\!\frac{1}{R_+-R}
+ O(1)
\qquad (R\uparrow R_+).
\]
\end{lemma}

\begin{proof}
We prove the statement for $Z_-$; the proof for $Z_+$ is analogous, with the
orientation of the endpoint integral reversed. Recall from
\Cref{lem:Z-solution-formula} that
\[
Z_-(R)=\mu_-(R)^{-1}\int_{R_-}^{R}\mu_-(s)Q(s)\,ds,
\]
where
\[
\mu_-(R):=\exp\!\left(\int_{R_-^\ast}^{R}P(s)\,ds\right)
\]
for an arbitrary $R_-^\ast\in(R_-,R_s)$.

Since $\mathcal{S}'(R_-)\neq 0$ (see \Cref{lem:S-two-roots}), the coefficient
$P$ is continuous near $R_-$. Hence $\mu_-$ extends as a $C^1$ function to
$R_-$, and
\[
\mu_-(R)=\mu_-(R_-)+O(R-R_-)
\qquad (R\downarrow R_-).
\]
Moreover,
\[
\mu_-(R_-)
=
\exp\!\left(\lim_{R\downarrow R_-}\int_{R_-^\ast}^{R}P(s)\,ds\right)>0,
\]
because the integral has a finite limit.

Since $\mathcal{S}(R_-)=0$, $\mathcal{S}'(R_-)\neq0$, and $\mathcal{S}$ is
$C^2$ near $R_-$, Taylor's theorem gives
$\mathcal{S}(R)=\mathcal{S}'(R_-)(R-R_-)+O\!\left((R-R_-)^2\right)$
and
$\mathcal{S}'(R)=\mathcal{S}'(R_-)+O(R-R_-)$.
Therefore
\[
\frac{\mathcal{S}(R)}{\mathcal{S}'(R)}
=
(R-R_-)+O\!\left((R-R_-)^2\right),
\]
and hence
\[
Q(R)=\frac{2}{\alpha}(R-R_-)+O\!\left((R-R_-)^2\right)
\qquad (R\downarrow R_-).
\]
Combining this with the expansion of $\mu_-$ yields
\[
\mu_-(R)Q(R)
=
\frac{2\mu_-(R_-)}{\alpha}(R-R_-)
+O\!\left((R-R_-)^2\right).
\]
Thus
\[
\int_{R_-}^{R}\mu_-(s)Q(s)\,ds
=
\frac{\mu_-(R_-)}{\alpha}(R-R_-)^2
+O\!\left((R-R_-)^3\right).
\]
Dividing by
\[
\mu_-(R)=\mu_-(R_-)+O(R-R_-)
\]
gives
\[
Z_-(R)=\frac{1}{\alpha}(R-R_-)^2
+O\!\left((R-R_-)^3\right).
\]

Since $Z_->0$ on $(R_-,R_s)$, this implies
\[
\sqrt{Z_-(R)}
=
\frac{R-R_-}{\sqrt{\alpha}}\bigl(1+O(R-R_-)\bigr),
\]
and therefore
\[
\frac{1}{\sqrt{Z_-(R)}}
=
\frac{\sqrt{\alpha}}{R-R_-}+O(1)
\qquad (R\downarrow R_-).
\]
Splitting the integral at a fixed point sufficiently close to $R_-$ and absorbing
the remaining bounded contribution into $O(1)$ gives
\[
\int_R^{R_0^-}\frac{d\rho}{\sqrt{Z_-(\rho)}}
=
\sqrt{\alpha}\,\log\!\frac{1}{R-R_-}+O(1).
\]
The proof at $R_+$ is the same, using the formula for $Z_+$ and the variable
$R_+-R$ near the endpoint.
\end{proof}

\begin{proposition}[One-sided reconstruction of the canonical branches]\label{prop:reconstruction}
Fix $\widehat R_-\in(R_-,R_s)$ and $\widehat R_+\in(R_s,R_+)$. Define
\[
\Xi_-(R):=\int_{\widehat R_-}^{R}\frac{d\widetilde R}{\sqrt{Z_-(\widetilde R)}},
\qquad R\in(R_-,R_s),
\]
and
\[
\Xi_+(R):=\int_{\widehat R_+}^{R}\frac{d\widetilde R}{\sqrt{Z_+(\widetilde R)}},
\qquad R\in(R_s,R_+).
\]
Then $\Xi_\pm\in C^1$, with
\[
\Xi_\pm'(R)=\frac{1}{\sqrt{Z_\pm(R)}}>0.
\]
Thus each $\Xi_\pm$ is strictly increasing and invertible onto an open interval. Their inverses define
strictly increasing one-sided profile branches
\[
R^{\ell}(\xi):=\Xi_-^{-1}(\xi),
\qquad
\xi\in \Xi_-((R_-,R_s)),
\]
and
\[
R^{r}(\xi):=\Xi_+^{-1}(\xi),
\qquad
\xi\in \Xi_+((R_s,R_+)),
\]
satisfying
\[
\dot R^{\ell}(\xi)=\sqrt{Z_-(R^{\ell}(\xi))},
\qquad
\dot R^{r}(\xi)=\sqrt{Z_+(R^{r}(\xi))}.
\]
Moreover,
\[
\lim_{R\to R_-^+}\Xi_-(R)=-\infty,
\qquad
\lim_{R\to R_+^-}\Xi_+(R)=+\infty.
\]
\end{proposition}

\begin{proof}
We give the argument for a generic branch. Let $I=(R_-,R_s)$ with $Z=Z_-$, or
$I=(R_s,R_+)$ with $Z=Z_+$. By \Cref{lem:Z-solution-formula}, $Z\in C^1(I)$
and $Z>0$ on $I$. Hence the corresponding map $\Xi$ is $C^1$ and satisfies
\[
\Xi'(R)=\frac{1}{\sqrt{Z(R)}}>0.
\]
Therefore $\Xi$ is strictly increasing and maps $I$ bijectively onto the open
interval $\Xi(I)$. By the inverse function theorem, $\Xi^{-1}$ is $C^1$ on
$\Xi(I)$. If $R(\xi):=\Xi^{-1}(\xi)$, then differentiating
$\Xi(R(\xi))=\xi$ gives
$\dot R(\xi)=
\sqrt{Z(R(\xi))}$.
The endpoint limits follow from the logarithmic estimates in
\Cref{lem:Z-endpoints}. On the left branch,
\[
\Xi_-(R)
=
-\int_R^{\widehat R_-}\frac{d\rho}{\sqrt{Z_-(\rho)}}
\to -\infty
\qquad (R\downarrow R_-),
\]
while on the right branch,
\[
\Xi_+(R)
=
\int_{\widehat R_+}^{R}\frac{d\rho}{\sqrt{Z_+(\rho)}}
\to +\infty
\qquad (R\uparrow R_+).
\]
\end{proof}

Applied to the canonical one-sided solutions $Z_-$ and $Z_+$ from
\Cref{lem:Z-solution-formula}, \Cref{prop:reconstruction} produces the left and
right profile branches away from the sonic density. The remaining task is to
analyze the behavior of $Z_\pm$ as $R\to R_s^\pm$, reconstruct the branches up
to the sonic point, and glue them into a global weak profile; this is carried
out in \Cref{sec:global}.

%============================================================================================
\section{Global existence, uniqueness, and regularity}
\label{sec:global}

\Cref{sec:structure-shock} constructed the canonical one-sided solutions $Z_-$ and $Z_+$,
representing the squared slope of the density profile, together with the corresponding
one-sided reconstruction maps away from the sonic density. The remainder of the argument
is organized in three steps. First, we quantify the singular behavior of $Z_\pm$ as
$R\to R_s^\pm$. Second, we reconstruct the profile on each side up to the sonic point
and glue the two branches into a global weak density profile. 
Third, we show that the one-sided representation also yields uniqueness up to 
translation. Finally, we compile these results to extract the resulting regularity and decay properties 
and reconstruct the full traveling-wave fields.

The sonic asymptotics are governed by the exponent
\begin{equation}\label{eq:a-def}
a_s:=-\frac{2p''(R_s)}{\mathcal{S}''(R_s)}\ge 0 .
\end{equation}
Recall, $R_s\in(R_-,R_+)$ is the unique sonic density, characterized by
$\mathcal{S}'(R_s)=0$. 
Recalling the definition of the coefficient $P$ \eqref{eq:PQ},
this exponent arises from the leading-order singular part
of $P$ near $R_s$. 
By the standing convexity hypothesis $p''\ge0$, the nonzero
mass flux condition $m\neq0$, and \Cref{eq:Sdoubleprime}, one has
$\mathcal{S}''(R_s)
=
-\frac{2m^2}{R_s^3}-p''(R_s)<0$,
and hence $0\le a_s<2$. 

%--------------------------------------------------------------------------------------------
\subsection{Asymptotics near the sonic density}

\begin{lemma}[Sonic asymptotics of $Z$]\label{lem:Z-sonic-asympt}
Let $Z_\pm$ be the canonical one-sided solutions from
\Cref{lem:Z-solution-formula}, and let $I_-:=(R_-,R_s)$ and
$I_+:=(R_s,R_+)$.
Then $Z_\pm(R)\to+\infty$ as $R\to R_s^\pm$. More precisely:
\begin{enumerate}[label=\textup{(\roman*)}]
\item If $a_s>0$, equivalently $p''(R_s)>0$, then there exist
$c_\pm,C_\pm>0$ and $\eta>0$ such that, for
$R\in I_\pm$ with $0<|R-R_s|<\eta$,
\begin{equation}\label{eq:Z-two-sided-power}
c_\pm |R-R_s|^{-a_s}
\le
Z_\pm(R)
\le
C_\pm |R-R_s|^{-a_s}.
\end{equation}
\item If $a_s=0$, equivalently $p''(R_s)=0$, then there exist
$c_\pm,C_\pm>0$ and $\eta>0$ such that, for
$R\in I_\pm$ with $0<|R-R_s|<\eta$,
\begin{equation}\label{eq:Z-two-sided-log}
c_\pm |\log |R-R_s||
\le
Z_\pm(R)
\le
C_\pm |\log |R-R_s||.
\end{equation}
\end{enumerate}
\end{lemma}

\begin{proof}
Define
\[
b_s:=\frac{2}{\alpha}\frac{\mathcal S(R_s)}{\mathcal S''(R_s)}<0.
\]
Indeed, $\mathcal S(R_s)>0$ by \Cref{lem:S-two-roots}, while
$\mathcal S''(R_s)<0$ by \Cref{eq:Sdoubleprime}.

We first record the coefficient expansions near $R_s$. Since
$\mathcal S'(R_s)=0$ and $p''$ is locally Lipschitz near $R_s$, Taylor expansion gives,
one-sidedly as $R\to R_s^\pm$,
\[
\mathcal S(R)=\mathcal S(R_s)+O(|R-R_s|^2),
\]
and
\[
\mathcal S'(R)=\mathcal S''(R_s)(R-R_s)+O(|R-R_s|^2).
\]
Hence
\[
Q(R)=\frac{2}{\alpha}\frac{\mathcal S(R)}{\mathcal S'(R)}
=
\frac{b_s}{R-R_s}+O(1)
\qquad (R\to R_s^\pm).
\]
Here and below, all $O(\cdot)$ estimates are understood one-sidedly on the branch under consideration.

\emph{Case $a_s>0$.}
In this case $p''(R_s)>0$. Since $p''$ is locally Lipschitz, we have 
$p''(R)=p''(R_s)+O(|R-R_s|)$.
Using this and the definition of $P$ \eqref{eq:PQ},
we obtain
\[
P(R)=\frac{a_s}{R-R_s}+O(1)
\qquad (R\to R_s^\pm),
\]
where $a_s=-2p''(R_s)/\mathcal S''(R_s)>0$.

The singular term $a_s/(R-R_s)$ in $P$ suggests factoring out the corresponding
power-law behavior from the integrating factor. This is simply the
variation-of-constants formula with the leading sonic singularity removed.
Define
\[
\mu_{0,\pm}(R)
:=
\exp\!\left(
\int_{R_s}^{R}
\left(P(s)-\frac{a_s}{s-R_s}\right)\,ds
\right),
\qquad R\in I_\pm,
\]
where the integral is interpreted one-sidedly. Since $P(R)-\frac{a_s}{R-R_s}=O(1)$,
each $\mu_{0,\pm}$ extends continuously to $R_s$ and is bounded above and below
by positive constants near $R_s$.

Define
\[
M_\pm(R):=|R-R_s|^{a_s}\mu_{0,\pm}(R),
\qquad R\in I_\pm.
\]
This is the integrating factor for the equation after the singular part
$a_s/(R-R_s)$ of $P$ has been separated off. Indeed,
\[
\frac{M_\pm'(R)}{M_\pm(R)}
=
\frac{a_s}{R-R_s}
+
\frac{\mu_{0,\pm}'(R)}{\mu_{0,\pm}(R)}
=
\frac{a_s}{R-R_s}
+
P(R)-\frac{a_s}{R-R_s}
=
P(R).
\]
Thus, if we set
\[
Y_\pm(R):=M_\pm(R)Z_\pm(R)
=
|R-R_s|^{a_s}\mu_{0,\pm}(R)Z_\pm(R),
\]
then multiplying $Z_\pm'+PZ_\pm=Q$ by $M_\pm$ gives
\[
Y_\pm'(R)=M_\pm(R)Q(R)
=
|R-R_s|^{a_s}\mu_{0,\pm}(R)Q(R).
\]
Since $\mu_{0,\pm}$ is bounded near $R_s$ and
$Q(R)=\frac{b_s}{R-R_s}+O(1)$,
we have
\[
Y_\pm'(R)=O(|R-R_s|^{a_s-1}).
\]
Because $a_s>0$, this bound is integrable on each one-sided neighborhood of
$R_s$. Hence $Y_\pm$ has a finite one-sided limit
$L_\pm:=\lim_{R\to R_s^\pm}Y_\pm(R)$.
Equivalently,
\[
Z_\pm(R)
=
|R-R_s|^{-a_s}\mu_{0,\pm}(R)^{-1}Y_\pm(R),
\]
so the behavior of $Z_\pm$ is determined by whether the renormalized amplitude
$Y_\pm$ has a nonzero limit at the sonic point.

We claim that $L_\pm>0$. Since $Z_\pm>0$ by
\Cref{lem:Z-solution-formula} and the prefactor defining $Y_\pm$ is positive,
one has $Y_\pm>0$ on the corresponding branch. Therefore its finite one-sided
limit satisfies $L_\pm\ge0$. Suppose $L_\pm=0$. Then
\[
Y_\pm(R)
=
\int_{R_s}^{R}|s-R_s|^{a_s}\mu_{0,\pm}(s)Q(s)\,ds.
\]
Using
\[
Q(s)=\frac{b_s}{s-R_s}+O(1),
\qquad
\mu_{0,\pm}(s)=\mu_{0,\pm}(R_s)+o(1),
\]
and the one-sided identity
\[
\int_{R_s}^{R}\frac{|s-R_s|^{a_s}}{s-R_s}\,ds
=
\frac{1}{a_s}|R-R_s|^{a_s},
\]
we obtain, on either branch,
\[
Y_\pm(R)
=
\frac{b_s\,\mu_{0,\pm}(R_s)}{a_s}|R-R_s|^{a_s}
+
o(|R-R_s|^{a_s})
\qquad (R\to R_s^\pm).
\]
Since $b_s<0$, this is negative for $R$ sufficiently close to $R_s$, contradicting
$Y_\pm>0$. Hence $L_\pm>0$.

Therefore, after possibly shrinking $\eta>0$, both $Y_\pm$ and $\mu_{0,\pm}$ are
bounded above and below by positive constants for $R\in I_\pm$ with
$0<|R-R_s|<\eta$. It follows from
\[
Z_\pm(R)
=
|R-R_s|^{-a_s}\mu_{0,\pm}(R)^{-1}Y_\pm(R)
\]
that there exist $c_\pm,C_\pm>0$ such that
\[
c_\pm |R-R_s|^{-a_s}
\le
Z_\pm(R)
\le
C_\pm |R-R_s|^{-a_s},
\qquad
0<|R-R_s|<\eta.
\]
This proves \Cref{eq:Z-two-sided-power}, and in particular
$Z_\pm(R)\to+\infty$ as $R\to R_s^\pm$.

\emph{Case $a_s=0$.}
Here $p''(R_s)=0$. Since $p''$ is locally Lipschitz, $p''(R)=O(|R-R_s|)$. Together with
\[
\mathcal S'(R)=\mathcal S''(R_s)(R-R_s)+O(|R-R_s|^2),
\]
this gives $\frac{p''(R)}{\mathcal S'(R)}=O(1)$,
and hence, $P(R)=O(1)$ as $R\to R_s^\pm$.

In contrast with the case $a_s>0$, there is no power-law singularity to factor
out of the integrating factor. We therefore use the ordinary one-sided
integrating factor
\[
\mu_\pm(R):=\exp\!\left(\int_{R_s}^{R}P(s)\,ds\right),
\qquad R\in I_\pm.
\]
Since $P$ is bounded near $R_s$, each $\mu_\pm$ extends continuously to $R_s$
and is bounded above and below by positive constants. Moreover,
\[
\mu_\pm(R)=\mu_\pm(R_s)+O(|R-R_s|).
\]

Define $Y_\pm(R):=\mu_\pm(R)Z_\pm(R)$. Since
\[
\frac{\mu_\pm'(R)}{\mu_\pm(R)}=P(R),
\]
multiplying $Z_\pm'+PZ_\pm=Q$ by $\mu_\pm$ gives
\[
Y_\pm'(R)=\mu_\pm(R)Q(R).
\]
Using
\[
Q(R)=\frac{b_s}{R-R_s}+O(1)
\quad \text{ and } \quad
\mu_\pm(R)=\mu_\pm(R_s)+O(|R-R_s|),
\]
we obtain
\[
Y_\pm'(R)
=
\frac{b_s\mu_\pm(R_s)}{R-R_s}+O(1)
\qquad (R\to R_s^\pm).
\]
Therefore,
\[
Y_\pm(R)
=
b_s\mu_\pm(R_s)\log|R-R_s|+O(1)
\qquad (R\to R_s^\pm).
\]
Recalling that $Y_\pm=\mu_\pm Z_\pm$, this becomes
\[
\mu_\pm(R)Z_\pm(R)
=
b_s\mu_\pm(R_s)\log|R-R_s|+O(1)
\qquad (R\to R_s^\pm).
\]
Since $b_s<0$ and $\log|R-R_s|\to-\infty$, the leading term is positive and
diverges like $|\log|R-R_s||$. Thus, after possibly shrinking $\eta>0$, there
exist constants $\widetilde c_\pm,\widetilde C_\pm>0$ such that
\[
\widetilde c_\pm |\log|R-R_s||
\le
\mu_\pm(R)Z_\pm(R)
\le
\widetilde C_\pm |\log|R-R_s||,
\qquad
0<|R-R_s|<\eta.
\]
Since $\mu_\pm$ is bounded above and below by positive constants near $R_s$,
there exist $c_\pm,C_\pm>0$ such that
\[
c_\pm |\log|R-R_s||
\le
Z_\pm(R)
\le
C_\pm |\log|R-R_s||,
\qquad
0<|R-R_s|<\eta.
\]
This proves \Cref{eq:Z-two-sided-log}, and in particular
$Z_\pm(R)\to+\infty$ as $R\to R_s^\pm$.
\end{proof}

\begin{remark}[Polytropic special case]\label{rem:poly-special-a}
For $p(R)=\kappa R^\gamma$ with $\kappa>0$, $\gamma\ge 1$, from \Cref{eq:a-def}, one computes
\[
a_s=\frac{2(\gamma-1)}{\gamma+1}.
\]
Hence $a_s>0$ exactly when $\gamma>1$, while $a_s=0$ for $\gamma=1$.

The same formula holds for the ideal-gas-type equation of state obtained from the reduced pressure:
\[
\tilde p(R,e)=(\gamma-1)Re,
\quad \text{where} \quad
e(R)=e_0+\frac{A}{R}+\frac{m^2}{2R^2}.
\]
Indeed,
\[
p(R)=\tilde p(R,e(R))
=(\gamma-1)\left(e_0R+A+\frac{m^2}{2R}\right),
\]
so $p''(R)=(\gamma-1)m^2/R^3$.
Evaluating at $R_s$ and using $\mathcal S''(R_s)=-\frac{2m^2}{R_s^3}-p''(R_s)$ \eqref{eq:Sdoubleprime}
yields
\[
\mathcal S''(R_s)=-(\gamma+1)\frac{m^2}{R_s^3},
\qquad
a_s=-\frac{2p''(R_s)}{\mathcal S''(R_s)}
=\frac{2(\gamma-1)}{\gamma+1}.
\]
Thus, for a calorically perfect gas, the sonic exponent again satisfies
\[
a_s=\frac{2(\gamma-1)}{\gamma+1}.
\]
\end{remark}

\begin{remark}[Reconstruction integrability at $R_s$]\label{rem:sonic-integrability}
\Cref{lem:Z-sonic-asympt} implies that
$1/\sqrt{Z_\pm(R)}\in L^1_{\mathrm{loc}}$
near $R_s$ on each one-sided branch. Indeed, if $a_s>0$, then $1/\sqrt{Z_\pm(R)}\le C_\pm |R-R_s|^{a_s/2}$,
which is locally integrable. If $a_s=0$, then
\[
\frac{1}{\sqrt{Z_\pm(R)}}\le
\frac{C_\pm}{\sqrt{|\log |R-R_s||}},
\]
which is also locally integrable near $R_s$.
\end{remark}

%--------------------------------------------------------------------------------------------
\subsection{Reconstruction and gluing of the density profile}

Applying \Cref{prop:reconstruction} to the canonical one-sided solutions $Z_-$ and
$Z_+$ produces strictly increasing left and right profile branches, each determined
up to translation in $\xi$. By \Cref{lem:Z-endpoints}, these branches extend to
$\xi=\pm\infty$ at the outer end states, while \Cref{rem:sonic-integrability}
shows that they reach the sonic density $R_s$ in finite $\xi$.

Now, one must check that the glued profile satisfies the profile equation
distributionally, rather than acquiring a point-mass defect at the crossing.
Although each branch solves
the profile equation classically away from $R_s$, the density slope may blow up
at the sonic crossing. The following lemma rules out the formation of such a defect.

\begin{lemma}[No defect at the sonic crossing]\label{lem:no-defect}
Let $R:\mathbb{R}\to [R_-,R_+]$ be continuous, with $R(\xi_s)=R_s$, and suppose
$R\in C^2((-\infty,\xi_s)\cup(\xi_s,\infty))$.
Assume that
\[
R((-\infty,\xi_s))\subset (R_-,R_s),
\qquad
R((\xi_s,\infty))\subset (R_s,R_+),
\]
and that
\[
\dot R=\sqrt{Z_-(R)}
\quad\text{on }(-\infty,\xi_s),
\qquad
\dot R=\sqrt{Z_+(R)}
\quad\text{on }(\xi_s,\infty),
\]
where $Z_\pm$ are the canonical one-sided solutions from
\Cref{lem:Z-solution-formula}. Then
\[
R\in W^{1,2}_{\mathrm{loc}}(\mathbb R),
\]
and $R$ satisfies \Cref{eq:R-ODE-general} on $\mathbb R$ in the sense of distributions, i.e., \Cref{eq:weak-profile}.
\end{lemma}

\begin{proof}
\emph{Step 1: local Sobolev regularity near the sonic point.}
Away from $\xi_s$, the profile is $C^2$. 

It remains only to check the behavior near $\xi_s$. On each side of $\xi_s$,
\[
\dot R^2=Z_\pm(R),
\qquad
d\xi=\frac{dR}{\sqrt{Z_\pm(R)}}.
\]
Hence, for $\epsilon>0$ sufficiently small,
\begin{multline*}
\int_{\xi_s-\epsilon}^{\xi_s+\epsilon} |\dot R(\xi)|^2\,d\xi
=
\int_{\xi_s-\epsilon}^{\xi_s} |\dot R(\xi)|^2\,d\xi
+
\int_{\xi_s}^{\xi_s+\epsilon} |\dot R(\xi)|^2\,d\xi \\
=
\int_{R(\xi_s-\epsilon)}^{R_s}\sqrt{Z_-(R)}\,dR
+
\int_{R_s}^{R(\xi_s+\epsilon)}\sqrt{Z_+(R)}\,dR.
\end{multline*}
Both integrals are finite by \Cref{lem:Z-sonic-asympt}: in the power-law case
this uses $a_s<2$, and in the logarithmic case
$\sqrt{|\log|R-R_s||}$ is locally integrable. Since $R$ is bounded, this proves
$R\in W^{1,2}_{\mathrm{loc}}(\mathbb R)$.

\emph{Step 2: classical validity away from the sonic point.}
On each side of $\xi_s$, the identity $\dot R^2=Z_\pm(R)$ holds and $Z_\pm$
solves \Cref{eq:Z-general}. Reversing the reduction in
\Cref{prop:Z-equation}, the profile satisfies \Cref{eq:R-ODE-general}
pointwise on $(-\infty,\xi_s)\cup(\xi_s,\infty)$.
Indeed, differentiating $\dot R^2=Z_\pm(R)$ gives
$\ddot R=\frac12 Z_\pm'(R)$, and substituting this into
\Cref{eq:Z-general} recovers \Cref{eq:R-ODE-expanded}.

\emph{Step 3: vanishing of the possible defect.}
Write \Cref{eq:R-ODE-general} in divergence form as
$G-\alpha F'=0$,
where, away from $\xi_s$,
\[
F(\xi):=\frac{\mathcal S'(R(\xi))}{R(\xi)}\dot R(\xi),
\qquad
G(\xi):=
\frac{\mathcal S(R(\xi))}{R(\xi)}
-2\alpha\frac{m^2}{R(\xi)^4}\dot R(\xi)^2.
\]
Since the equation holds classically on each side, the only possible
distributional defect at $\xi_s$ is a point mass coming from a jump in $F$.

Near $R_s$, $\mathcal S'(R)=\mathcal S''(R_s)(R-R_s)+O(|R-R_s|^2)$
and $\dot R=\sqrt{Z_\pm(R)}$. Hence, \Cref{lem:Z-sonic-asympt} gives
\[
|F(\xi)|
\le C
\begin{cases}
|R(\xi)-R_s|^{1-a_s/2}, & a_s>0,\\[0.5ex]
|R(\xi)-R_s|\sqrt{|\log |R(\xi)-R_s||}, & a_s=0.
\end{cases}
\]
Since $0\le a_s<2$, both right-hand sides tend to zero as
$\xi\to\xi_s^\pm$. Therefore
\[
\lim_{\xi\to\xi_s^-}F(\xi)
=
\lim_{\xi\to\xi_s^+}F(\xi)
=
0.
\]
Thus the divergence term produces no point-mass defect at the sonic crossing.

We also have $F,G\in L^1_{\mathrm{loc}}(\mathbb R)$. The estimate above gives
local integrability of $F$, while $G$ is locally integrable because
$\mathcal S(R)/R$ is bounded near $\xi_s$ and $\dot R^2\in L^1_{\mathrm{loc}}$
by Step 1.

\emph{Step 4: passage to the weak formulation.}
Let $\varphi\in C_c^\infty(\mathbb R)$ and $\varepsilon>0$. Since
$G-\alpha F'=0$ holds classically on each side of $\xi_s$, integration by parts
on $(-\infty,\xi_s-\varepsilon)$ and $(\xi_s+\varepsilon,\infty)$ gives
\begin{align*}
&\int_{\mathbb R\setminus[\xi_s-\varepsilon,\xi_s+\varepsilon]}
\alpha F\dot{\varphi}\,d\xi
+
\int_{\mathbb R\setminus[\xi_s-\varepsilon,\xi_s+\varepsilon]}
G\varphi\,d\xi \\
&\qquad
+\alpha F(\xi_s+\varepsilon)\varphi(\xi_s+\varepsilon)
-\alpha F(\xi_s-\varepsilon)\varphi(\xi_s-\varepsilon)
=0.
\end{align*}
Letting $\varepsilon\to0$, the bulk integrals converge by local integrability,
and the boundary terms vanish because $F(\xi)\to0$ from both sides. Hence
\[
\int_{\mathbb R}\alpha F\dot{\varphi}\,d\xi
+
\int_{\mathbb R}G\varphi\,d\xi
=0.
\]
This is exactly \eqref{eq:weak-profile}.
\end{proof}

\begin{theorem}[Global weak density profile by gluing]\label{thm:global-glue}
Under the standing assumptions \Cref{assump:main}, there exists a continuous strictly
increasing profile $R:\mathbb{R}\to(R_-,R_+)$ and a crossing point $\xi_s\in\mathbb R$
such that:
\begin{enumerate}[label=\textup{(\roman*)}]
\item $R(\xi_s)=R_s$ and $R$ restricts to a $C^2$ function on
$(-\infty,\xi_s)\cup(\xi_s,\infty)$;
\item $\dot R^2=Z_-(R)\text{ on }(-\infty,\xi_s), \dot R^2=Z_+(R)\text{ on }(\xi_s,\infty)$;
\item
$R(\xi)\to R_\pm$ and 
$\dot R(\xi)\to0$ as $\xi\to\pm\infty$;
\item $R$ satisfies the weak formulation \Cref{eq:weak-profile}; 
hence, together with monotonicity and the endpoint limits, $R\in\mathcal M$.
\end{enumerate}
\end{theorem}

\begin{proof}
\emph{Step 1: reconstruct the one-sided branches.}
Apply \Cref{prop:reconstruction} to the canonical one-sided solutions $Z_-$ and
$Z_+$. This gives strictly increasing branches $R^\ell$ and $R^r$ satisfying
\[
\dot R^\ell=\sqrt{Z_-(R^\ell)},
\qquad
\dot R^r=\sqrt{Z_+(R^r)}.
\]
Since $Z_\pm\in C^1$ and $Z_\pm>0$ on their domains, the reconstructed branches
are $C^2$ away from the sonic density.

\emph{Step 2: identify the domains and glue.}
By \Cref{lem:Z-endpoints}, the left branch extends to $\xi=-\infty$ with limit
$R_-$, while the right branch extends to $\xi=+\infty$ with limit $R_+$.
By \Cref{rem:sonic-integrability}, both branches reach $R_s$ in finite $\xi$.
After translating the two branches, we may assume that both reach $R_s$ at
$\xi_s=0$.

Define
\[
R(\xi):=
\begin{cases}
R^\ell(\xi), & \xi<0,\\
R_s, & \xi=0,\\
R^r(\xi), & \xi>0.
\end{cases}
\]
Then $R$ is continuous and strictly increasing, with $R(0)=R_s$. The identities
in \textup{(ii)} hold by construction, and the endpoint limits
$R(\xi)\to R_\pm$ follow from the preceding paragraph.

\emph{Step 3: verify the weak equation.}
The hypotheses of \Cref{lem:no-defect} are satisfied by the glued profile.
Therefore,
$R\in W^{1,2}_{\mathrm{loc}}(\mathbb R)$
and $R$ satisfies \Cref{eq:R-ODE-general} on $\mathbb R$ in the sense of
distributions. Since $R$ is nondecreasing and connects $R_-$ to $R_+$, it follows
that $R\in\mathcal M$. This proves \textup{(iv)}.

\emph{Step 4: endpoint flattening.}
Now that $R\in\mathcal M$, \Cref{lem:tail-flattening} applies and gives
\[
\dot R(\xi)\to0
\qquad\text{as }\xi\to\pm\infty.
\]
This completes \textup{(iii)} and the proof.
\end{proof}

%--------------------------------------------------------------------------------------------
\subsection{Regularity and decay of the glued density profile}

\begin{proposition}[Sonic regularity of the glued profile]\label{prop:sonic-regularity}
Let $R$ be the global profile furnished by \Cref{thm:global-glue}. Then:
\begin{enumerate}[label=(\alph*)]
\item if $0<a_s<2$, then $R\in C^{0,\beta}$ near $\xi_s$ with
\[
\beta=\frac{1}{1+a_s/2}\in\left(\frac12,1\right),
\]
$R\in W^{1,r}_{\mathrm{loc}}$ for every
$1\le r<1+\frac{2}{a_s}$, and $|\dot{R}(\xi)|\to\infty$ as $\xi\to\xi_s^\pm$;
\item if $a_s=0$, then $R\in C^{0,\beta}$ near $\xi_s$ for every $\beta<1$,
$R\in W^{1,r}_{\mathrm{loc}}$ for every finite $r\ge 1$, and $|\dot{R}(\xi)|\to\infty$ as $\xi\to\xi_s^\pm$.
\end{enumerate}
In particular, in both cases the density profile is not locally Lipschitz near $\xi_s$.
\end{proposition}

\begin{proof}
Away from $\xi_s$, the profile is $C^2$ by \Cref{thm:global-glue}; 
hence all regularity assertions are local questions near the sonic crossing. On each branch,
\[
\frac{d\xi}{dR}=\frac{1}{\sqrt{Z_\pm(R)}}.
\]
If $a_s>0$, then \Cref{eq:Z-two-sided-power} implies there exist $c_1,C_1>0$ such that
\[
c_1|R-R_s|^{a_s/2}\le \frac{d\xi}{dR}\le C_1|R-R_s|^{a_s/2}
\]
near $R_s$. Integrating yields constants $c_2,C_2>0$ with
\[
c_2|\xi-\xi_s|^\beta\le |R-R_s|\le C_2|\xi-\xi_s|^\beta,
\qquad \beta=(1+a_s/2)^{-1}.
\]
Therefore, for suitable $c_3,C_3>0$,
\[
c_3|\xi-\xi_s|^{\beta-1}\le |\dot{R}(\xi)|\le C_3|\xi-\xi_s|^{\beta-1},
\qquad 0<|\xi-\xi_s|<\delta,
\]
where $\dot R$ denotes the classical derivative on the two $C^2$ branches. 
These local estimates give the stated H\"older behavior near $\xi_s$. Moreover, $|\dot R(\xi)|^r \lesssim |\xi-\xi_s|^{r(\beta-1)}$,
which is locally integrable precisely when $r(1-\beta)<1$. Since
\[
\frac{1}{1-\beta}=1+\frac{2}{a_s},
\]
we obtain $R\in W^{1,r}_{\mathrm{loc}}$ for every
$1\le r<1+\frac{2}{a_s}$. Since $\beta<1$, one also has
$|\dot R(\xi)|\to\infty$ as $\xi\to\xi_s^\pm$, so $R$ cannot be locally Lipschitz at $\xi_s$.

If $a_s=0$, then \Cref{eq:Z-two-sided-log} gives
\[
\frac{c_1}{\sqrt{|\log|R-R_s||}}
\le \frac{d\xi}{dR}
\le
\frac{C_1}{\sqrt{|\log|R-R_s||}}
\]
near $R_s$. Integrating gives
\[
|\xi-\xi_s|
\gtrsim
\int_0^{|R-R_s|}
\frac{ds}{\sqrt{|\log s|}}
\gtrsim
\frac{|R-R_s|}{\sqrt{|\log|R-R_s||}},
\]
because, for $x>0$ sufficiently small,
\[
\int_0^x \frac{ds}{\sqrt{|\log s|}}
\ge
\int_{x/2}^{x} \frac{ds}{\sqrt{|\log s|}}
\ge
\frac{1}{\sqrt{2|\log x|}}
\int_{x/2}^{x} ds
=
\frac{x}{2\sqrt{2|\log x|}}.
\]
Equivalently, for every $\beta<1$ and $|R-R_s|$ sufficiently small,
\[
|R-R_s|\le C_\beta |\xi-\xi_s|^\beta,
\]
because
\[
|R-R_s|^{1-\beta}|\log|R-R_s||^{\beta/2}\to0
\qquad (R\to R_s).
\]
These local estimates give the stated H\"older behavior near $\xi_s$. Moreover,
\[
\int |\dot R|^r\,d\xi
=
\int Z_\pm(R)^{(r-1)/2}\,dR
\lesssim
\int |\log|R-R_s||^{(r-1)/2}\,dR,
\]
which is finite for every finite $r\ge 1$. Finally, on the two $C^2$ branches,
\[
|\dot R(\xi)|=\sqrt{Z_\pm(R(\xi))}\to\infty
\qquad \text{as } \xi\to\xi_s^\pm,
\]
so $R$ is not locally Lipschitz at $\xi_s$.
\end{proof}

\begin{lemma}[Exponential tails from endpoint expansions]\label{lem:tail-decay}
Let $R$ be the global profile from \Cref{thm:global-glue}. For every
$\lambda \in (0, 1/\sqrt{\alpha})$, there exist constants $C,\xi_0>0$ such that
\[
|R(\xi)-R_+|+|\dot{R}(\xi)|\le C e^{-\lambda \xi}
\qquad (\xi\ge \xi_0),
\]
and
\[
|R(\xi)-R_-|+|\dot{R}(\xi)|\le C e^{\lambda \xi}
\qquad (\xi\le -\xi_0).
\]
In particular, $\alpha^{-1/2}$ is the linearized exponential tail rate on both tails.
\end{lemma}

\begin{proof}
Choose $\xi_0$ large enough so that $\xi\ge \xi_0$ lies on the right branch and
$\xi\le -\xi_0$ lies on the left branch. Set
\[
y_-(\xi):=R(\xi)-R_- \qquad (\xi<\xi_s),
\qquad
y_+(\xi):=R_+-R(\xi) \qquad (\xi>\xi_s).
\]
Then $y_\pm>0$ and
\[
\dot y_-(\xi)=\sqrt{Z_-(R(\xi))},
\qquad
-\dot y_+(\xi)=\sqrt{Z_+(R(\xi))}.
\]
By \Cref{lem:Z-endpoints},
\[
\frac{\sqrt{Z_\pm(R)}}{|R-R_\pm|}
\to
\frac{1}{\sqrt{\alpha}}
\qquad (R\to R_\pm).
\]
Fix $\lambda\in(0,\alpha^{-1/2})$. Then, after increasing $\xi_0$ if necessary,
there exists $\mu>\alpha^{-1/2}$ such that
\[
\lambda y_\pm(\xi)\le |\dot y_\pm(\xi)|\le \mu y_\pm(\xi)
\]
on the corresponding tails.

On the right tail, $\dot y_+\le -\lambda y_+$, hence
\[
y_+(\xi)\le y_+(\xi_0)e^{-\lambda(\xi-\xi_0)}
\qquad (\xi\ge \xi_0).
\]
On the left tail, $\dot y_-\ge \lambda y_-$, and integrating from $\xi$ to
$-\xi_0$ gives
\[
y_-(\xi)\le y_-(-\xi_0)e^{\lambda(\xi+\xi_0)}
\qquad (\xi\le -\xi_0).
\]
Finally, $|\dot R|=|\dot y_\pm|\le \mu y_\pm$ on the same tails, so the stated
bounds follow after absorbing fixed prefactors into $C$.
\end{proof}

\begin{corollary}[Global decay and Sobolev consequences]\label{cor:global-profile-regularity}
Let $R$ be the global profile from \Cref{thm:global-glue}.
\begin{enumerate}[label=\textup{(\alph*)}]
\item If $0<a_s<2$, then $\dot{R}\in L^r(\mathbb{R})$ for every
$1\le r<1+\frac{2}{a_s}$. If $a_s=0$, then $\dot{R}\in L^r(\mathbb{R})$ for
every finite $r\ge 1$.
\item If $R_{\mathrm{int}}\in C^\infty(\mathbb{R})$ satisfies
\[
R_{\mathrm{int}}(\xi)=R_-
\quad\text{for }\xi\le \xi_s-1,
\qquad
R_{\mathrm{int}}(\xi)=R_+
\quad\text{for }\xi\ge \xi_s+1,
\]
then
\[
R-R_{\mathrm{int}}\in L^\mu(\mathbb{R})
\qquad \text{for all }\mu\in[1,\infty],
\]
and, for every $r$ in the range specified in \textup{(a)},
\[
R-R_{\mathrm{int}}\in W^{1,r}(\mathbb{R}).
\]
\end{enumerate}
\end{corollary}

\begin{proof}
Local integrability near $\xi_s$ is the content of
\Cref{prop:sonic-regularity}. Away from a fixed neighborhood of $\xi_s$,
\Cref{lem:tail-decay} gives exponential decay of $|\dot R|$, hence
$\dot R\in L^r$ there for every $r\ge1$. Combining the near-sonic and tail
regions proves \textup{(a)}.

For \textup{(b)}, on the tails $R_{\mathrm{int}}$ is constant, so
$R-R_{\mathrm{int}}=R-R_\pm$, which decays exponentially by
\Cref{lem:tail-decay}. On the remaining compact region, $R-R_{\mathrm{int}}$ is
continuous and hence belongs to $L^\mu$ for every $\mu\in[1,\infty]$. Therefore
$R-R_{\mathrm{int}}\in L^\mu(\mathbb{R})$ for all such $\mu$.

Moreover,
$\dot R-\dot R_{\mathrm{int}}\in L^r(\mathbb{R})$
because $\dot R\in L^r(\mathbb{R})$ by \textup{(a)} and
$\dot R_{\mathrm{int}}\in C_c^\infty(\mathbb{R})$. Since also
$R-R_{\mathrm{int}}\in L^r(\mathbb{R})$, it follows that
$R-R_{\mathrm{int}}\in W^{1,r}(\mathbb{R})$.
\end{proof}

%--------------------------------------------------------------------------------------------
\subsection{Uniqueness and reconstructed traveling-wave fields}

\begin{theorem}[Uniqueness modulo translation]\label{thm:unique-translation}
Assume the hypotheses of \Cref{thm:global-glue} hold for the fixed parameters
\((c,m,q,\alpha,p)\). Let \(R,\widetilde R:\mathbb{R}\to(R_-,R_+)\) be two global profiles
connecting \(R_-\) to \(R_+\) in the sense of \Cref{thm:global-glue}. Then there exists \(\xi_0\in\mathbb{R}\) such that
\[
    \widetilde R(\xi)=R(\xi-\xi_0)\qquad(\xi\in\mathbb{R}).
\]
If both profiles are normalized by the phase condition
\(R(0)=\widetilde R(0)=R_s\), then \(\widetilde R\equiv R\).
\end{theorem}

\begin{proof}
Let $\xi_s$ and $\widetilde\xi_s$ be the sonic crossing points of $R$ and
$\widetilde R$, respectively, and set
$I_-:=(R_-,R_s)$, $I_+:=(R_s,R_+).$
Since both profiles are strictly increasing, their inverse maps
$\Psi_\pm,\widetilde\Psi_\pm:I_\pm\to\mathbb R$
are well defined on the two density intervals. By \Cref{thm:global-glue}, on each
branch one has
\[
\dot R^2=Z_\pm(R),
\qquad
\dot{\widetilde R}^{\,2}=Z_\pm(\widetilde R).
\]
Since the profiles are increasing, this gives
\[
\Psi_\pm'(\rho)
=
\frac{1}{\sqrt{Z_\pm(\rho)}}
=
\widetilde\Psi_\pm'(\rho),
\qquad \rho\in I_\pm.
\]
Hence, there exist constants $c_\pm$ such that $\widetilde\Psi_\pm(\rho)=\Psi_\pm(\rho)+c_\pm$.

Taking the limit $\rho\to R_s^\pm$ gives $c_\pm=\widetilde\xi_s-\xi_s$.
Thus $c_-=c_+=:\xi_0$, so the inverse maps differ by the same additive constant
on both branches. Therefore,
$\widetilde R(\xi)=R(\xi-\xi_0)$
on both sides of the sonic point, and the identity holds at the sonic point by
continuity. If $R(0)=\widetilde R(0)=R_s$, then
$\xi_s=\widetilde\xi_s=0$, so $\xi_0=0$ and $\widetilde R\equiv R$.
\end{proof}

We next transfer these properties from the density profile to the reconstructed
traveling-wave fields.

\begin{theorem}[Full solution and its regularity]\label{thm:full-solution}
Let $R$ denote the global profile furnished by \Cref{thm:global-glue}, unique up to
translation by \Cref{thm:unique-translation}. Define
\[
U(\xi)=c+\frac{m}{R(\xi)},\qquad
e(\xi)=e(R(\xi)),\qquad
S(\xi)=\mathcal S(R(\xi)),
\]
and
\[
\mathcal E(\xi)=R(\xi)e(\xi)+\frac12 R(\xi)U(\xi)^2.
\]
Then:
\begin{enumerate}[label=\textup{(\roman*)}]
\item $(R,U,\mathcal E,S)$ is a weak solution of
\Crefrange{eq:TWa}{eq:TWd} on $\mathbb R$.
\item $R$, $U$, $\mathcal E$, $e$, and $S$ restrict to $C^2$ functions on
$(-\infty,\xi_s)\cup(\xi_s,\infty)$.
\item $R$, $U$, $\mathcal E$, $e$, and $S$ have the same local H\"older regularity
near $\xi_s$ as $R$ in \Cref{prop:sonic-regularity}; namely, if $0<a_s<2$ they
are $C^{0,\beta}$ with $\beta=(1+a_s/2)^{-1}$, while if $a_s=0$ they are
$C^{0,\beta}$ for every $\beta<1$.
\item $R$, $U$, $\mathcal E$, $e$, and $S$ belong to $L^\infty(\mathbb R)$.
Moreover, if $R_{\mathrm{int}}$ is as in
\Cref{cor:global-profile-regularity} and takes values in $[R_-,R_+]$, and if
$U_{\mathrm{int}}$, $\mathcal E_{\mathrm{int}}$, $e_{\mathrm{int}}$, and
$S_{\mathrm{int}}$ are obtained from $R_{\mathrm{int}}$ by the same
reconstruction formulas, then
\[
R-R_{\mathrm{int}},\quad U-U_{\mathrm{int}},\quad
\mathcal E-\mathcal E_{\mathrm{int}},\quad e-e_{\mathrm{int}},\quad
S-S_{\mathrm{int}}
\]
belong to $L^\mu(\mathbb R)$ for all $\mu\in[1,\infty]$.
\item For every $r$ in the range specified in
\Cref{cor:global-profile-regularity}, $\dot{R}, \dot{U}, \dot{\mathcal{E}}, \dot{e}, \dot{S} \in L^r(\mathbb{R})$.
\item For the same range of $r$,
\[
R-R_{\mathrm{int}},\quad U-U_{\mathrm{int}},\quad
\mathcal E-\mathcal E_{\mathrm{int}},\quad e-e_{\mathrm{int}},\quad
S-S_{\mathrm{int}}
\]
belong to $W^{1,r}(\mathbb R)$.
\end{enumerate}
\end{theorem}

\begin{proof}
For the density component $R$, the asserted regularity, decay, and Sobolev
properties follow from
\Cref{thm:global-glue,prop:sonic-regularity,cor:global-profile-regularity}.

For the reconstructed fields, write
\[
U=F_U(R),\qquad e=F_e(R),\qquad
\mathcal E=F_{\mathcal E}(R),\qquad S=F_S(R),
\]
where
\[
F_U(\rho)=c+\frac{m}{\rho},\qquad
F_e(\rho)=e(\rho),\qquad
F_S(\rho)=\mathcal S(\rho),
\]
and
\[
F_{\mathcal E}(\rho)=\rho F_e(\rho)+\frac12 \rho F_U(\rho)^2.
\]
Under the standing assumptions, these reconstruction maps are $C^2$ on a
neighborhood of $[R_-,R_+]$. Since $R(\xi)\in[R_-,R_+]$, the maps and their
first derivatives are bounded on the range of $R$.

Since the reconstruction maps are $C^2$ on a neighborhood of $[R_-,R_+]$,
composition with $R$ gives \textup{(ii)} on the nonsonic intervals. Moreover,
these maps are Lipschitz on $[R_-,R_+]$. Hence, for any reconstructed field
$f=F(R)$,
\[
|f(\xi_1)-f(\xi_2)|
\le
\|F'\|_{L^\infty([R_-,R_+])}|R(\xi_1)-R(\xi_2)|,
\]
which transfers the H\"older regularity of $R$ to $f$ and proves
\textup{(iii)}. Boundedness of the same reconstruction maps on $[R_-,R_+]$
also gives the $L^\infty$ part of \textup{(iv)}.

Finally, defining $f_{\mathrm{int}}:=F(R_{\mathrm{int}})$, the same Lipschitz
bound gives
\[
|f-f_{\mathrm{int}}|
\le
\|F'\|_{L^\infty([R_-,R_+])}|R-R_{\mathrm{int}}|.
\]
Therefore \Cref{cor:global-profile-regularity} implies
$f-f_{\mathrm{int}}\in L^\mu(\mathbb R)$ for all $\mu\in[1,\infty]$, completing
\textup{(iv)}.

For \textup{(v)}, the chain rule gives
$\dot f(\xi)=F'(R(\xi))\dot R(\xi)$.
Hence,
\[
|\dot f(\xi)|
\le
\sup_{\rho\in[R_-,R_+]}|F'(\rho)|\,|\dot R(\xi)|,
\]
and the claimed $L^r$ integrability follows from
\Cref{cor:global-profile-regularity}.

For \textup{(vi)}, we already know $f-f_{\mathrm{int}}\in L^r(\mathbb R)$.
Moreover,
\[
\dot f-\dot f_{\mathrm{int}}
=
F'(R)\dot R-F'(R_{\mathrm{int}})\dot R_{\mathrm{int}}.
\]
The first term is in $L^r$ by \textup{(v)}, while the second is in $L^r$ because
$F'(R_{\mathrm{int}})$ is bounded and $\dot R_{\mathrm{int}}$ is smooth with
compact support. Thus $f-f_{\mathrm{int}}\in W^{1,r}(\mathbb R)$.

It remains to verify the weak traveling-wave equations. The flux identities
\[
R(U-c)=m,
\qquad
mU+p(R)+S=q,
\qquad
(\mathcal E+p(R)+S)U-c\mathcal E=k
\]
hold pointwise by construction, so the weak forms of
\Crefrange{eq:TWa}{eq:TWc} follow immediately. Finally, the closure equation
\Cref{eq:TWd} is exactly the weak profile equation for $R$ after substituting
\[
S(\xi)=\mathcal S(R(\xi)),
\qquad
\dot U(\xi)=-\frac{m}{R(\xi)^2}\dot R(\xi),
\]
which is valid a.e. Since $R\in\mathcal M$ by \Cref{thm:global-glue}, this weak
profile equation holds. Hence $(R,U,\mathcal E,S)$ is a weak solution of
\Crefrange{eq:TWa}{eq:TWd}.
\end{proof}

\section{Small-\texorpdfstring{$\alpha$}{alpha} asymptotics and the $\alpha \rightarrow 0^+$ limit}\label{sec:asymptotics}

In this section, we study the small-$\alpha$ asymptotics of the IGR compressive shock solution and in particular, discuss in what sense the solution converges to the entropy-admissible Euler shock as $\alpha \rightarrow 0^+$. The main tool for the analysis will be to show that the shock profile can be expressed as a scaling of an $\alpha$-independent profile, from which several results follow immediately.
In particular, the spatiotemporal rescaling $(x,t) \mapsto (x/\sqrt{\alpha}, t/\sqrt{\alpha})$ renders the IGR Euler system, given in \Cref{eq:IGR}, parameter independent, effectively setting $\alpha = 1$. This observation is the basis for the analysis in this section. 

%-------------------------------------------------------------

\subsection{Scaling structure of the $Z$ equation}

Unlike the previous discussion where the $\alpha$ dependence was implicit, here, for any $\alpha>0$, we will explicitly denote $\alpha$-dependent quantities with a subscript $\alpha$. Throughout this subsection, $Z_\alpha$ is understood branchwise on either nonsonic density interval, with the corresponding outer endpoint condition. Recall from \Cref{prop:Z-equation} that $Z_\alpha(R)$ satisfies the linear ODE
\begin{equation}
Z'_\alpha + P(R) Z_\alpha = Q_\alpha(R),\ 
P(R)=2\!\left(
\frac{\mathcal S''(R)}{\mathcal S'(R)}
-\frac{1}{R}
+\frac{2m^2}{R^3\mathcal S'(R)}
\right),\
Q_\alpha(R)=\frac{2}{\alpha}\frac{\mathcal S(R)}{\mathcal S'(R)}.
\label{eq:Z-alpha}
\end{equation}
The dependence on $\alpha$ enters only through the prefactor in $Q_\alpha$. We can therefore obtain $Z_\alpha$ as an appropriate rescaling of an $\alpha$-independent solution.

\begin{proposition}[$Z_\alpha$ from rescaling of an $\alpha$-independent profile]
For $\alpha>0$, let $Z_\alpha$ denote the canonical one-sided solution of \Cref{eq:Z-alpha}. Then
\[
Z_\alpha(R) = \alpha^{-1}\,\mathcal Z(R),
\]
where $\mathcal Z$ is the corresponding canonical one-sided solution of the $\alpha$-independent problem
\begin{equation}
\mathcal Z' + P(R)\mathcal Z
=
2\frac{\mathcal S(R)}{\mathcal S'(R)}.
\label{eq:Z-reduced}
\end{equation}
Equivalently, $\mathcal Z=Z_1$. The outer endpoint condition is inherited branchwise: on the left branch $\mathcal Z(R_-)=0$, while on the right branch $\mathcal Z(R_+)=0$.
\end{proposition}

\begin{proof}
Substituting the ansatz $Z_\alpha=\alpha^{-1}\mathcal Z$ into
\Cref{eq:Z-alpha} gives \Cref{eq:Z-reduced}, and the corresponding outer endpoint condition is preserved. Uniqueness of the canonical one-sided solutions follows from \Cref{lem:Z-solution-formula}.
\end{proof}

%-------------------------------------------------------------

\subsection{Shock width scaling}

The above rescaling applies to both one-sided solutions and can be glued together to produce a global profile $R_\alpha$, noting that the sonic density $R_s$ is independent of $\alpha$. Particularly, this profile satisfies $\dot{R}_\alpha = \sqrt{Z_\alpha(R)}$, which combined with the previous result gives
\begin{equation}\label{eq:alpha-dR-dxi}
\dot R_\alpha(\xi)
=
\frac{1}{\sqrt{\alpha}}\sqrt{\mathcal Z(R_\alpha(\xi))}.
\end{equation}
Hence,
\begin{equation}\label{eq:alpha-dxi-dR}
\frac{d\xi}{dR_\alpha}
=
\sqrt{\alpha}\,
\frac{1}{\sqrt{\mathcal Z(R)}}.
\end{equation}

\begin{proposition}[Shock width scaling]
Let $R_\alpha: \mathbb{R} \rightarrow (R_-, R_+)$ denote the IGR shock profile with inverse $R_\alpha^{-1}: (R_-, R_+) \rightarrow \mathbb{R}$. For any
$\delta\in(0,(R_+-R_-)/2)$, define
\[
\xi_\alpha^-(\delta)=R_\alpha^{-1}(R_-+\delta),
\qquad
\xi_\alpha^+(\delta)=R_\alpha^{-1}(R_+-\delta).
\]

Then
\[
\xi_\alpha^+(\delta)-\xi_\alpha^-(\delta)
=
\sqrt{\alpha}\,C_\delta
\]
for a constant $C_\delta$ depending on $\delta$ but independent of $\alpha$. Thus, the shock width scales like $\sqrt{\alpha}$.
\end{proposition}

\begin{proof}
Integrating the previous relation \Cref{eq:alpha-dxi-dR} gives
\[
\xi_\alpha^+(\delta)-\xi_\alpha^-(\delta)
=
\sqrt{\alpha}
\int_{R_-+\delta}^{R_+-\delta}
\frac{dR}{\sqrt{\mathcal Z(R)}}.
\]
Since $\delta > 0$, the integral is finite (c.f. \Cref{sec:one-sided-reconstruct}) and depends only on $\delta$.
\end{proof}

Another way to view that the shock width scales like $\sqrt{\alpha}$ is that the profile $R_\alpha$ can be expressed as a self-similar $(1/\sqrt{\alpha})$-scaling of an $\alpha$-independent profile $\mathcal{R}$. To see this, let $\xi_s^\alpha$ denote the sonic crossing point where $R_\alpha(\xi_s^\alpha)=R_s$. Define the rescaled coordinate
$\eta:=(\xi-\xi_s^\alpha)/\sqrt{\alpha}$.
Then, we have the following self-similar rescaling.

\begin{proposition}[Self--similar profile scaling]
There exists a profile $\mathcal R$ independent of $\alpha$
such that
\begin{equation}\label{eq:alpha-independent-R}
    R_\alpha(\xi)
    =
    \mathcal R\!\left(
    \frac{\xi-\xi_s^\alpha}{\sqrt{\alpha}}
    \right).
\end{equation}

\end{proposition}

\begin{proof}
Note that \Cref{eq:alpha-independent-R} defines $\mathcal{R}(\eta)$ to be $R_\alpha(\xi)$; thus, we just need to see that $\mathcal{R}$ is independent of $\alpha$. Combining \Cref{eq:alpha-independent-R,eq:alpha-dR-dxi}, and the chain rule gives
$ d\mathcal{R}/d\eta=\sqrt{\mathcal Z(\mathcal{R})}$,
which determines $\mathcal R$ independently of $\alpha$.
\end{proof}

%-------------------------------------------------------------
\subsection{Euler shock limit}

We now show that the IGR traveling-waves converge to the classical Euler shock as $\alpha\to0^+$. First, we establish a locally uniform pointwise limit, away from the sonic crossing. For the remainder, to remove the translational degree of freedom, we will assume that the sonic crossing occurs at the origin.

\begin{theorem}[Pointwise Euler shock limit away from the sonic crossing]
Let $(R_\alpha,U_\alpha,e_\alpha,S_\alpha)$ denote the IGR traveling-wave. Then, 
\[
R_\alpha(\xi)\to
\begin{cases}
R_- & \xi<0,\\
R_+ & \xi>0,
\end{cases}
\]
as $\alpha\to0^+$. The convergence holds locally uniformly on
$(-\infty,0) \cup (0,\infty)$ with exponential rate. Analogous statements hold for $U_\alpha$, $e_\alpha$ and $S_\alpha$.
\end{theorem}

\begin{proof}
Since the shock width scales like $\sqrt{\alpha}$ about the origin, for fixed $\xi\neq0$ the solution lies in the exponential tail region for sufficiently small $\alpha$. \Cref{lem:tail-decay} then gives the claimed locally uniform convergence of $R_\alpha$ to the asymptotic states at an exponential rate. Similarly, the locally uniform convergence of $U_\alpha,e_\alpha,S_\alpha$ follows from the relations in \Cref{eq:U-from-R,eq:S-as-function-of-R,eq:e-explicit}.
\end{proof}

\begin{remark}
Note that, since the asymptotic states satisfy the Rankine--Hugoniot relations and the Lax transonic inequalities, the limiting discontinuity coincides with the entropy Euler shock connecting the left and right states. 
\end{remark}

Now, we establish convergence of the IGR shock solution to the Euler shock limit in $L^\mu(\mathbb{R})$, $\mu \in [1,\infty)$. Let $R_0$ denote the above pointwise limit, i.e., $R_0$ equals $R_-$ for $\xi < 0$ and $R_+$ for $\xi > 0$. 

\begin{theorem}[Convergence to the Euler shock in $L^\mu(\mathbb{R})$]\label{thm:euler-shock-Lmu}
    Let $\mu\in[1,\infty)$. Then, the $\alpha$-IGR shock solution converges to $R_0$ in $L^\mu(\mathbb{R})$ at rate $\alpha^{1/(2\mu)}$, i.e.,
    \begin{equation}\label{eq:Lp-convergence-to-Euler}
        \|R_\alpha - R_0\|_{L^\mu(\mathbb{R})} \leq  C \alpha^{1/(2\mu)},
    \end{equation}
    where $C$ is independent of $\alpha$. An analogous statement holds for $U_\alpha,e_\alpha,S_\alpha$. 
\end{theorem}
\begin{proof}
    This follows from a direct calculation using the self-similar profile \Cref{eq:alpha-independent-R} and change of variables $\xi \rightarrow \eta$.
\end{proof}

\begin{remark}
    See \Cref{fig:igr_tw_alpha_convergence} for a numerical verification of the convergence rate of the IGR shock profile to the Euler shock in $L^1$ and $L^2$. 
\end{remark}

 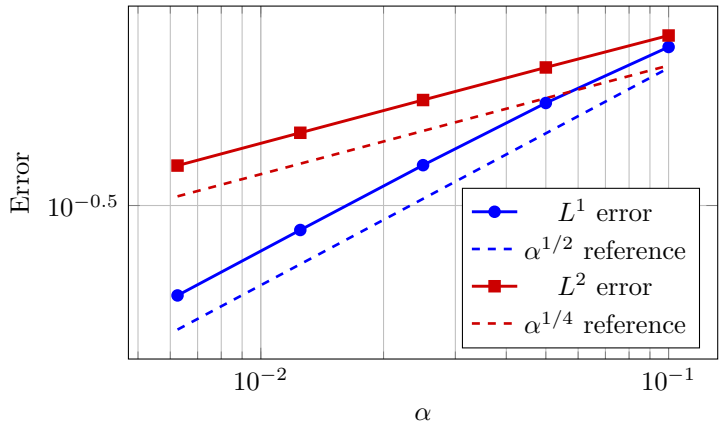
\begin{figure}[h]
     \centering
     \IfFileExists{figures/data/igr_tw_alpha_convergence.csv}{%
  \def\igralphacsv{figures/data/igr_tw_alpha_convergence.csv}%
}{%
  \def\igralphacsv{data/igr_tw_alpha_convergence.csv}%
}
\def\igrrefshift{0.85}

\begin{tikzpicture}
\begin{axis}[
    width=0.72\textwidth,
    height=0.48\textwidth,
    xmode=log,
    ymode=log,
    grid=both,
    xlabel={$\alpha$},
    ylabel={Error},
    legend style={font=\normalsize, at={(0.97,0.03)}, anchor=south east},
    tick label style={font=\normalsize},
    label style={font=\normalsize},
    title style={font=\normalsize},
]

\addplot[
    blue,
    line width=1.1pt,
    mark=*,
    mark size=1.8pt,
] table [x=alpha, y=L1_error, col sep=comma] {\igralphacsv};

\addplot[
    blue,
    dashed,
    line width=1pt,
] table [x=alpha, y expr=\igrrefshift*\thisrow{L1_ref}, col sep=comma] {\igralphacsv};

\addplot[
    red!80!black,
    line width=1.1pt,
    mark=square*,
    mark size=1.8pt,
] table [x=alpha, y=L2_error, col sep=comma] {\igralphacsv};

\addplot[
    red!80!black,
    dashed,
    line width=1pt,
] table [x=alpha, y expr=\igrrefshift*\thisrow{L2_ref}, col sep=comma] {\igralphacsv};

\legend{$L^1$ error, $\alpha^{1/2}$ reference, $L^2$ error, $\alpha^{1/4}$ reference}

\end{axis}
\end{tikzpicture}
     \caption{Numerical convergence of IGR traveling-wave profiles to the Euler shock profile as $\alpha \to 0^+$.}
     \label{fig:igr_tw_alpha_convergence}
 \end{figure}

As a corollary, $L^\mu(\mathbb{R})$ convergence of the solution implies convergence of the derivative in a suitable negative Sobolev norm: for $\mu\in[1,\infty)$, denote by $W^{-1,\mu}(\mathbb{R})$ the topological dual of the closure $W^{1,\mu'}_0(\mathbb{R})$ of $C^\infty_c(\mathbb{R})$ in $W^{1,\mu'}(\mathbb{R})$, where $\mu' = \mu/(\mu-1)$ (where $\mu=1$ is understood to correspond to $\mu'=\infty$) equipped with the operator norm
\begin{equation}
    \| \chi \|_{W^{-1,\mu}(\mathbb{R})} := \sup_{\substack{\phi \in W^{1,\mu'}_0(\mathbb{R}), \\ \|\phi\|_{W^{1,\mu'}(\mathbb{R})}=1}} |\langle \chi, \phi\rangle|,
\end{equation}
where $\langle\cdot,\cdot\rangle: W^{-1,\mu}(\mathbb{R}) \times W^{1,\mu'}_0(\mathbb{R}) \rightarrow \mathbb{R}$ denotes the duality pairing. 

Let $\delta_0$ denote the delta distribution centered at the origin. Note that $\delta_0 \in W^{-1,\mu}(\mathbb{R})$ for $\mu \in [1,\infty)$ since by Sobolev embedding, $W^{1,\mu'}(\mathbb{R}) \hookrightarrow C^0(\mathbb{R})$ for any $\mu' \in (1,\infty]$, which yields continuous pointwise evaluation. In particular, the distributional derivative $(R_+-R_-)\delta_0$ of the Euler shock $R_0$ is in $W^{-1,\mu}(\mathbb{R})$.

\begin{corollary}[Strong convergence of the derivative in $W^{-1,\mu}(\mathbb{R})$]
    Let $\mu \in [1,\infty)$. Then, $\{\dot{R}_\alpha d\xi\}_{\alpha > 0}$ converges to $(R_+ - R_-)\delta_0$ in $W^{-1,\mu}(\mathbb{R})$ as $\alpha \rightarrow 0^+$ with rate $\alpha^{1/(2\mu)}$. In particular,
    \begin{equation}
        \| \dot{R}_\alpha d\xi - (R_+ - R_-)\delta_0 \|_{W^{-1,\mu}(\mathbb{R})} \leq C \alpha^{1/(2\mu)},
    \end{equation}
    where $C$ is the same constant as in \Cref{thm:euler-shock-Lmu}. Analogous statements hold for $\dot{U}_\alpha d\xi$, $\dot{e}_\alpha d\xi$, $\dot{S}_\alpha d\xi$.
    \begin{proof} This immediately follows from a density argument using H\"{o}lder's inequality and \Cref{thm:euler-shock-Lmu}.
    \end{proof}
\end{corollary}

    Note that for the finite $\mu$ considered, $W^{1,\mu'}_0(\mathbb{R})$ is a strict subset of $W^{1,\mu'}(\mathbb{R})$ only for the endpoint $\mu'=\infty$ (i.e., $\mu = 1$), due to loss of control of mass at infinity. We can instead establish convergence of the derivative in a space smaller than $W^{-1,1}(\mathbb{R})$ at the cost of trading strong convergence for weak$^*$ convergence, which is the next result.

\begin{theorem}[Weak convergence of the derivative in the sense of measures]\label{thm:weak-convergence-measure}
    The sequence of Radon measures $\{\dot{R}_\alpha d\xi\}_{\alpha > 0}$ converges to $(R_+-R_-)\delta_0$ as $\alpha\rightarrow 0^+$ weakly in the sense of measures, i.e., weak$^*$ convergence on the topological dual of $C_0(\mathbb{R})$. Analogous statements hold for $\dot{U}_\alpha d\xi$, $\dot{e}_\alpha d\xi$, and $\dot{S}_\alpha d\xi$.
    \begin{proof}
        We will use a density argument. First, consider any $\phi \in C^1_c(\mathbb{R})$; then,
        \begin{align*}
            \left|\int_{\mathbb{R}} \phi(\xi)\dot{R}_\alpha(\xi)d\xi - \int_{\mathbb{R}} \phi(\xi) (R_+-R_-)\delta_0\right| &= \left|-\int_{\mathbb{R}} \dot{\phi}(\xi) (R_\alpha(\xi) - R_0(\xi)) d\xi\right| \\
            &\leq \|\dot{\phi}\|_\infty \|R_\alpha - R_0\|_{L^1(\mathbb{R})} \leq C \sqrt{\alpha} \|\dot{\phi}\|_\infty,
        \end{align*}
        where we used \Cref{thm:euler-shock-Lmu}. For $\psi \in C_0(\mathbb{R})$, the remainder of the proof follows by a standard density argument, taking a sequence $\{\phi_n \in C^1_c(\mathbb{R})\}_{n=1}^\infty$ approximating $\psi$ in supremum norm.
    \end{proof}
\end{theorem}

\subsection{Continuous dependence on parameters} To conclude this section, we briefly note that the various quantities derived throughout depend continuously on the IGR shock parameters $\theta = (\alpha, R_-, R_+)$. To see this, view $\mathcal{S}$ as explicitly depending on these parameters, $\mathcal{S} = \mathcal{S}(R; \theta)$. Recall that the sonic density $R_s$ defined by $\frac{\partial\mathcal{S}}{\partial R}(R; \theta) = 0$ furthermore satisfies $\frac{\partial^2\mathcal{S}}{\partial R^2}(R; \theta) \neq 0$. By the implicit function theorem, we have that $R_s$ has local $C^1$ dependence on $\theta$. Since the standing assumptions define an open set of admissible parameters, we have that $R_s$ has $C^1$ dependence on $\theta$ (note that $R_s$ is independent of $\alpha$ as discussed above, so this can be simplified to $C^1$ dependence on $R_-$ and $R_+$). This then yields continuous dependence on $\theta$ for various quantities derived in the paper, for example the $\alpha$-independent profile $\mathcal{R}$, as well as $a_s$ defined in \Cref{eq:a-def}. Hence, the H\"{o}lder and Sobolev exponents depend continuously on $\theta$, as well as various constants appearing in bounds throughout, e.g., \Cref{eq:Lp-convergence-to-Euler}.

%============================================================================================
\section{Conclusion}

In this paper, we established existence, uniqueness modulo translation, and regularity of compressive transonic shock profiles in one dimension for the information geometric regularization of the compressible Euler equations with a general thermodynamic equation of state. The profiles are monotone heteroclinic connections between saddle equilibria and exhibit a degeneracy at the transition from a supersonic to a subsonic state. Despite the loss of classical regularity at the sonic point, we proved H\"older and Sobolev regularity and showed, in particular, that the global weak solution contains no singular measure defect. We also analyzed the small-regularization asymptotics and proved convergence to the entropy-admissible Euler shock as the regularization parameter tends to zero.
This work considers IGR shock profiles as traveling-wave solutions of the one-dimensional IGR equations, but does not address whether these profiles emerge dynamically from general initial data, whether they are stable under perturbations, or, if stable, the rate at which solutions approach the traveling profile. Prior work suggests that the divergence of the derivative at the sonic point may occur only asymptotically in time for smooth initial profiles \cite{cao2024information}. Moreover, one-dimensional IGR simulations exhibit profiles qualitatively similar to those obtained from the IGR shock equation in \Cref{fig:igr_tw_profile}. Future work will investigate the dynamical stability of IGR shock profiles both theoretically and numerically.

\section*{Acknowledgments}
Los Alamos National Laboratory Report LA-UR-26-25301.

\section*{Data availability statement}
Data generated in this study are available from the corresponding author upon reasonable request.

\section*{Declarations}
The authors declare no competing interests.

OpenAI's ChatGPT was used for language editing and assistance with manuscript preparation. The authors assume responsibility for all content.

\bibliographystyle{siamplain}
\bibliography{tw_igr_paper.bib}

%============================================================================================
\appendix
\crefalias{section}{appendix}

%============================================================================================
\section{Convexity of the reduced pressure from the physical EOS} \label{appendix:barotropic_eos}
\renewcommand{\theequation}{A.\arabic{equation}}
\renewcommand{\theHequation}{A.\arabic{equation}}
\setcounter{equation}{0}

This appendix records convenient criteria ensuring convexity of the reduced pressure
\[
p(R):=\tilde{p}\big(R,e(R)\big)
\]
along the traveling-wave (TW) internal-energy curve. Under the standing assumption $m\neq 0$,
the explicit formula \Cref{eq:e-explicit} shows that $e=e(R)$ is smooth for $R>0$, so the reduced
pressure $p$ inherits the regularity of the physical EOS pressure $\tilde p$. Convexity of $p$ is the
structural hypothesis used in Sections~2--4. The examples below show how this criterion applies to
ideal and stiffened gases and to cold $+$ thermal EOS.

In this appendix, $\tilde{p}(\rho,e)$ denotes the physical EOS pressure, while $p(R)$ denotes the reduced pressure along the TW curve.

%--------------------------------------------------------------------------------------------
\subsection{Traveling-wave energy curve and reduced pressure}

From Section~2, the traveling-wave reduction determines the internal-energy curve
$e=e(R)$ by \Cref{eq:e-explicit} and defines the reduced pressure by
\Cref{eq:reduced-pressure}. In this appendix we use these formulas to compute
$p''$ and derive convenient convexity criteria.

For later use we record the derivatives of $e(R)$:
\begin{equation}\label{eq:app-e-derivatives}
    e'(R)= -\frac{A}{R^2}-\frac{m^2}{R^3},
    \qquad
    e''(R)=\frac{2A}{R^3}+\frac{3m^2}{R^4}.
\end{equation}
Combining these gives the TW identity
\begin{equation}\label{eq:app-TW-identity}
    e''(R)+\frac{2}{R}e'(R)=\frac{m^2}{R^4}>0.
\end{equation}

%--------------------------------------------------------------------------------------------
\subsection{Second-derivative formula for the reduced pressure}

Differentiating \Cref{eq:reduced-pressure} and using the chain rule yields $p'(R)=\tilde{p}_\rho+\tilde{p}_e\,e'(R)$, and
\begin{equation}\label{eq:app-ptilde-second-basic}
    p''(R)=\tilde{p}_{\rho\rho}+2\tilde{p}_{\rho e}\,e'(R)+\tilde{p}_{ee}\,\big(e'(R)\big)^2+\tilde{p}_e\,e''(R),
\end{equation}
where all partial derivatives of $\tilde{p}$ are evaluated at $(\rho,e)=(R,e(R))$. Using \Cref{eq:app-TW-identity} to eliminate $e''$ from \Cref{eq:app-ptilde-second-basic} gives the equivalent representation
\begin{equation}\label{eq:app-ptilde-second-TW}
    p''(R)
    =
    \Big(\tilde{p}_{\rho\rho}+\frac{m^2}{R^4}\tilde{p}_e\Big)
    +2\Big(\tilde{p}_{\rho e}-\frac{1}{R}\tilde{p}_e\Big)e'(R)
    +\tilde{p}_{ee}\,\big(e'(R)\big)^2,
    \qquad (\rho,e)=(R,e(R)).
\end{equation}
Thus $p''(R)$ is a quadratic polynomial in $e'(R)$ with coefficients determined by the EOS and the flux $m$.

It is convenient to introduce the abbreviations
\begin{equation}\label{eq:app-B-delta}
    B(R):=\tilde{p}_{\rho\rho}+\frac{m^2}{R^4}\tilde{p}_e,
    \qquad
    \delta(R):=\tilde{p}_{\rho e}-\frac{1}{R}\tilde{p}_e,
    \qquad (\rho,e)=(R,e(R)),
\end{equation}
so that \Cref{eq:app-ptilde-second-TW} becomes
\begin{equation}\label{eq:app-ptilde-second-quadratic}
    p''(R)=B(R)+2\delta(R)e'(R)+\tilde{p}_{ee}\,\big(e'(R)\big)^2.
\end{equation}
Equivalently, defining the TW-modified Hessian
\begin{equation}\label{eq:app-Hcal}
    \mathcal{H}(R,e):=
    \begin{pmatrix}
        \tilde{p}_{\rho\rho}+\dfrac{m^2}{R^4}\tilde{p}_e & \tilde{p}_{\rho e}-\dfrac{1}{R}\tilde{p}_e\\[6pt]
        \tilde{p}_{\rho e}-\dfrac{1}{R}\tilde{p}_e & \tilde{p}_{ee}
    \end{pmatrix},
\end{equation}
we may write
\begin{equation}\label{eq:app-ptilde-second-matrix}
    p''(R)=
    \begin{pmatrix}1\\ e'(R)\end{pmatrix}^{\!T}
    \mathcal{H}\big(R,e(R)\big)
    \begin{pmatrix}1\\ e'(R)\end{pmatrix}.
\end{equation}

%--------------------------------------------------------------------------------------------
\subsection{A sufficient convexity criterion}

A simple sufficient condition for convexity of the reduced pressure on an interval of densities is positivity of the quadratic form in \Cref{eq:app-ptilde-second-matrix}.

\begin{proposition}[Sufficient condition for convexity]
\label{prop:app-convexity}
Assume that along the TW curve $(R,e(R))$ one has
\begin{equation}\label{eq:app-Hcal-psd}
    \mathcal{H}\big(R,e(R)\big)\succeq 0
\end{equation}
for all $R$ in the density range traversed by the profile. Then $p''(R)\ge 0$ on that interval,
i.e.\ $p$ is convex there.
\end{proposition}

Since $\mathcal{H}$ is a $2\times 2$ symmetric matrix, \Cref{eq:app-Hcal-psd} is equivalent to the scalar inequalities
\begin{equation}\label{eq:app-Hcal-psd-scalars}
    \tilde{p}_{ee}\ge 0,\qquad
    B(R)\ge 0,\qquad
    \delta(R)^2\le \tilde{p}_{ee}\,B(R),
    \qquad (\rho,e)=(R,e(R)),
\end{equation}
with $B$ and $\delta$ as in \Cref{eq:app-B-delta}. If $\tilde{p}_{ee}\equiv 0$ along the curve, then \Cref{eq:app-Hcal-psd-scalars} reduces to $\delta\equiv 0$ and $B(R)\ge 0$.

\begin{remark}[Interpretation of $\delta$]
\label{rem:app-delta}
Introducing the Gr\"uneisen parameter $\Gamma:=\tilde{p}_e/\rho$ \cite{gruneisen_param, riemann_fluid_flow}, one checks that
\begin{equation}\label{eq:app-delta-Gamma}
    \delta(R)=\tilde{p}_{\rho e}-\frac{1}{R}\tilde{p}_e
    = R\,\partial_\rho\Gamma(\rho,e)\big|_{e}\Big|_{\rho=R}.
\end{equation}
Thus $\delta$ measures the density dependence of the thermal coupling at fixed internal energy. In particular, $\delta\equiv 0$ corresponds to $\Gamma$ independent of $\rho$ (at fixed $e$).
\end{remark}

%--------------------------------------------------------------------------------------------
\subsection{Examples}

\paragraph{Ideal gas and stiffened gas}
For $\tilde{p}(\rho,e)=(\gamma-1)\rho e$ (and similarly $\tilde{p}=(\gamma-1)\rho e-\gamma p_\infty$), one has
$\tilde{p}_{ee}=0$, $\tilde{p}_{\rho e}=\gamma-1$, $\tilde{p}_e=(\gamma-1)\rho$. Hence, $\delta\equiv 0$ and
\[
B(R)=\frac{m^2}{R^4}\tilde{p}_e(R,e(R))=(\gamma-1)\frac{m^2}{R^3}>0.
\]
Therefore \Cref{eq:app-ptilde-second-quadratic} yields $p''(R)=(\gamma-1)\frac{m^2}{R^3}>0$, $R>0$,
so the reduced pressure is strictly convex along the TW curve despite $\mathrm{Hess}(\tilde{p})$ being indefinite.

\paragraph{Cold + thermal EOS}
Let $\tilde{p}(\rho,e)=p_0(\rho)+\rho\,\phi(e)$. Then $\tilde{p}_{\rho\rho}=p_0''(\rho)$, $\tilde{p}_{\rho e}=\phi'(e)$, $\tilde{p}_e=\rho\,\phi'(e)$, $\tilde{p}_{ee}=\rho\,\phi''(e)$,
and in particular $\delta\equiv 0$. Substituting into \Cref{eq:app-ptilde-second-TW} gives
\begin{equation}\label{eq:app-cold-thermal-ptilde-second}
    p''(R)
    =
    p_0''(R)
    +\phi'(e(R))\,\frac{m^2}{R^3}
    +R\,\phi''(e(R))\,\big(e'(R)\big)^2.
\end{equation}
Consequently, on any density interval traversed by the TW profile, $p$ is convex provided $p_0''(R)\ge 0$, $\phi'(e)\ge 0$, and $\phi''(e)\ge 0$, with strict convexity whenever $m\neq 0$ and $\phi'(e)>0$.

% %--------------------------------------------------------------------------------------------
% \subsection{Sign of \texorpdfstring{$A=mc-q$}{A = mc - q} at the end states}

% At constant end states of an IGR TW profile one has $S(\xi)\to 0$, hence from the integrated momentum balance
% \[
% q=mU_\pm+\tilde{p}(R_\pm,e_\pm)=mU_\pm+p(R_\pm).
% \]
% Using $U_\pm-c=m/R_\pm$ we obtain
% \begin{equation}\label{eq:app-A-sign}
%     A=mc-q
%     =m(c-U_\pm)-p(R_\pm)
%     =-\frac{m^2}{R_\pm}-p(R_\pm)\le 0,
% \end{equation}
% with strict negativity whenever either $m \neq 0$ (non-constant profile) or $p(R_\pm)>0$. 

%============================================================================================
\section{Alternative EOS admissibility criterion}
\label{appendix:eos-alt-criterion}
\renewcommand{\theequation}{B.\arabic{equation}}
\renewcommand{\theHequation}{B.\arabic{equation}}
\setcounter{equation}{0}
This appendix records a second, more direct admissibility criterion for reduced-pressure
convexity. Compared with \Cref{appendix:barotropic_eos}, it is typically easier to
check but more restrictive. In particular, it imposes the additional sign condition $e'(R)\ge 0$
along the TW curve, which is not part of the standing assumptions in the body of the paper.
As in \Cref{appendix:barotropic_eos}, $\tilde{p}(\rho,e)$ denotes the physical EOS pressure.

Starting from \Cref{eq:app-ptilde-second-TW},
\[
p''(R)
=
\Big(\tilde{p}_{\rho\rho}+\frac{m^2}{R^4}\tilde{p}_e\Big)
+2\Big(\tilde{p}_{\rho e}-\frac{1}{R}\tilde{p}_e\Big)e'(R)
+\tilde{p}_{ee}\,\big(e'(R)\big)^2,
\qquad (\rho,e)=(R,e(R)).
\]

\begin{proposition}[Alternative sufficient criterion]
\label{prop:app-alt-convexity}
Assume along the TW curve $(R,e(R))$ on the profile density interval that
\begin{equation}\label{eq:app-alt-conds}
    \tilde{p}_{\rho\rho}\ge 0,\qquad
    \tilde{p}_{ee}\ge 0,\qquad
    \tilde{p}_e\ge 0,\qquad
    \tilde{p}_{\rho e}\ge \frac{1}{R}\tilde{p}_e,\qquad
    e'(R)\ge 0.
\end{equation}
Then $p''(R)\ge 0$ on that interval. In particular, if $m\neq 0$ and $\tilde{p}_e>0$, then
$p''(R)>0$.
\end{proposition}

\begin{proof}
Under \Cref{eq:app-alt-conds}, each term in \Cref{eq:app-ptilde-second-TW} is nonnegative:
\[
\tilde{p}_{\rho\rho}+\frac{m^2}{R^4}\tilde{p}_e\ge 0,\qquad
2\Big(\tilde{p}_{\rho e}-\frac{1}{R}\tilde{p}_e\Big)e'(R)\ge 0,\qquad
\tilde{p}_{ee}\,\big(e'(R)\big)^2\ge 0.
\]
Hence $p''(R)\ge 0$. If additionally $m\neq 0$ and $\tilde{p}_e>0$, then
$\frac{m^2}{R^4}\tilde{p}_e>0$, so $p''(R)>0$.
\end{proof}

\begin{remark}[Growth interpretation]
Let $f(\rho,e):=\tilde{p}_e(\rho,e)$. The coupling condition
$\tilde{p}_{\rho e}\ge \frac{1}{R}\tilde{p}_e$ is equivalent to
\[
\partial_\rho f(\rho,e)\big|_{(\rho,e)=(R,e(R))}\ge \frac{f(R,e(R))}{R}.
\]
Thus, at the states sampled by the TW curve, the thermal response
$\tilde p_e$ must grow at least linearly in density when differentiated at fixed
internal energy.
\end{remark}

\begin{remark}[Separable EOS]
For $\tilde{p}(\rho,e)=g(e)\,h(\rho)$ with
$g\ge 0$, $g'\ge 0$, $g''\ge 0$, $h\ge0$, $h''\ge 0$, and
$\rho h'(\rho)\ge h(\rho)$, condition \Cref{eq:app-alt-conds} reduces to
$e'(R)\ge 0$ plus the stated sign/growth assumptions. In this case
\Cref{prop:app-alt-convexity} applies directly.
\end{remark}

%============================================================================================
\section{Numerical computation of the IGR shock profile} \label{appendix:numerics}
\renewcommand{\theequation}{C.\arabic{equation}}
\renewcommand{\theHequation}{C.\arabic{equation}}
\setcounter{equation}{0}

We compute IGR shock profiles by solving the scalar density equation on a truncated
interval $[-L,L]$ using a finite-difference damped-Newton method. For the numerical
examples we take a polytropic reduced pressure $p(R)=\kappa R^\gamma$, with
$\kappa>0$ and $\gamma\ge1$, and compute $(c,m,q)$ from RH-compatible end states
using \Cref{eq:wavespeed-and-fluxes}. The residual is formed from the expanded
density equation
\[
\frac{\mathcal{S}(R)}{R}
-\alpha\,\frac{\mathcal{S}'(R)}{R}\,\ddot{R}
-\alpha\left(\frac{\mathcal{S}''(R)}{R}
-\frac{\mathcal{S}'(R)}{R^2}
+\frac{2m^2}{R^4}\right)\dot{R}^{\,2}
=0.
\]

On a uniform grid $\xi_j=-L+jh$, $h=2L/N$, we impose
$R_0=R_-$ and $R_N=R_+$ and solve for the interior values
$\mathbf R=(R_1,\ldots,R_{N-1})$. At interior nodes we use the centered
approximations
\[
D_1R_j=\frac{R_{j+1}-R_{j-1}}{2h},
\qquad
D_2R_j=\frac{R_{j+1}-2R_j+R_{j-1}}{h^2},
\]
giving the nonlinear residual
\[
F_j(\mathbf R)
=
\frac{\mathcal S(R_j)}{R_j}
-\alpha\frac{\mathcal S'(R_j)}{R_j}D_2R_j
-\alpha\left(
\frac{\mathcal S''(R_j)}{R_j}
-\frac{\mathcal S'(R_j)}{R_j^2}
+\frac{2m^2}{R_j^4}
\right)(D_1R_j)^2.
\]

The nonlinear system $F(\mathbf R)=0$ is solved by damped Newton iteration with a
finite-difference Jacobian and residual-decreasing backtracking. Trial states are
constrained to remain positive. To select the physical branch, we use a projected
iteration: before residual comparison, trial profiles are projected onto the
monotone class, using cumulative maxima for increasing shocks and cumulative
minima for decreasing shocks, followed by clipping to the interval
$[R_-,R_+]$. For stronger shocks, we use continuation in the compression ratio
$r=R_+/R_-$. A sequence of increasing values of $r$ is solved, using each
converged profile as initial data for the next solve. Since the traveling wave is
translation-invariant, profiles are recentered between continuation steps so that
the midpoint level $\frac12(R_-+R_+)$ remains near the center of $[-L,L]$. 

\end{document}